\documentclass{amsart}

\usepackage{lipsum}
\usepackage{amsfonts}
\usepackage{amsmath}
\usepackage{amssymb}
\usepackage{amsthm}
\usepackage{hyperref}
\usepackage{epsf,latexsym,txfonts}
\usepackage{amsfonts}
\usepackage{fontawesome}
\usepackage{tikz}
\usetikzlibrary{arrows,matrix,positioning}
\usetikzlibrary{cd}
\usepackage{etex}
\usepackage{xy}
\usepackage{xspace}
\usepackage{mathrsfs}
\usepackage{graphicx}
\usepackage{epstopdf}
\usepackage{algorithmic}
\usepackage{verbatim}
\usepackage{ctable}
\ifpdf
\DeclareGraphicsExtensions{.eps,.pdf,.png,.jpg}
\else
\DeclareGraphicsExtensions{.eps}
\fi

\usepackage{enumitem}
\setlist[enumerate]{leftmargin=.5in}
\setlist[itemize]{leftmargin=.5in}

\newcommand{\wDelta}{\widehat{\Delta}}
\newcommand{\walpha}{\widehat{\alpha}}

\newcommand{\tS}{\mathtt{S}}

\newcommand{\tR}{\mathtt{R}}
\newcommand{\tJ}{\mathtt{J}}
\newcommand{\tI}{\mathtt{I}}

\newcommand{\R}{{\ensuremath{\mathbb{R}}}}
\newcommand{\calR}{{\ensuremath{\mathcal{R}}}}
\newcommand{\calJ}{{\ensuremath{\mathcal{J}}}}
\newcommand{\X}{{\ensuremath{\mathcal{X}}}}

\newcommand{\Proj}{{\ensuremath{\mathbb{P}}}}

\newcommand{\LBcs}{\ensuremath{\mathrm{LB}^{\text{\faStar}}}} 
\newcommand{\LBos}{\ensuremath{\mathrm{LB}^{\text{\faStarO}}}} 
\newcommand{\spl}{\ensuremath{\mathcal{S}}} 
\newcommand{\hspl}{\ensuremath{\mathcal{H}}} 
\newcommand{\inverseSystem}[1]{\ensuremath{{#1}^{\perp}}}

\DeclareMathOperator{\spn}{span}

\newtheorem{thm}{Theorem}[section]

\newtheorem{question}[thm]{Question}
\newtheorem{proposition}[thm]{Proposition}
\newtheorem{corollary}[thm]{Corollary}
\newtheorem{theorem}[thm]{Theorem}
\newtheorem{lemma}[thm]{Lemma}

\theoremstyle{definition}
\newtheorem{definition}[thm]{Definition}

\newtheorem{example}[thm]{Example}
\newtheorem{notation}[thm]{Notation}
\newtheorem{remark}[thm]{Remark}

\title{
	A lower bound for splines on tetrahedral vertex stars
}

\author[M. DiPasquale]{Michael DiPasquale}
\address{Michael DiPasquale\\     
	Department of Mathematics\\     
	Colorado State University}  
\email{michael.dipasquale@colostate.edu}
\urladdr{\url{https://midipasq.github.io/}}
\author[N. Villamizar]{Nelly Villamizar} 
\address{Nelly Villamizar\\
	Department of Mathematics\\
	Swansea University}
\email{n.y.villamizar@swansea.ac.uk}
\urladdr{\url{https://sites.google.com/site/nvillami}}

\begin{document}

\begin{abstract}
A tetrahedral complex all of whose tetrahedra meet at a common vertex is called a \textit{vertex star}.  Vertex stars are a natural generalization of planar triangulations, and understanding splines on vertex stars is a crucial step to analyzing trivariate splines.  It is particularly difficult to compute the dimension of splines on vertex stars in which the vertex is completely surrounded by tetrahedra -- we call these \textit{closed} vertex stars.  A formula due to Alfeld, Neamtu, and Schumaker gives the dimension of $C^r$ splines on closed vertex stars of degree at least $3r+2$.  We show that this formula is a lower bound on the dimension of $C^r$ splines of degree at least $(3r+2)/2$.  Our proof uses apolarity and the so-called \textit{Waldschmidt constant} of the set of points dual to the interior faces of the vertex star.  We also use an argument of Whiteley to show that the only splines of degree at most $(3r+1)/2$ on a generic closed vertex star are global polynomials.  
\end{abstract}


\maketitle

\section{Introduction}

A multivariate spline is a piecewise polynomial function on a partition $\Delta$ of some domain $\Omega\subset\R^n$ which is continuously differentiable to order $r$ for some integer $r\ge 0$.
Multivariate splines play an important role in many areas such as finite elements, computer-aided design, and data fitting~\cite{LaiSchumaker,Isogeometric}.  In these applications it is important to construct a basis, often with prescribed properties, for splines of bounded total degree.  A more basic task which aids in the construction of a basis is simply to compute the dimension of the space of multivariate splines of bounded degree on a fixed partition.  We write $\spl^r_d(\Delta)$ for the vector space of piecewise polynomial functions of degree at most $d$ on the partition $\Delta$ which are continuously differentiable of order $r$.

A formula for the dimension of $\spl^1_d(\Delta)$, where $\Delta$ is a planar triangulation, was first proposed by Strang~\cite{Strang} and proved for $d\ge 2$ by Billera~\cite{Homology}.  Subsequently the problem of computing the dimension of planar splines on triangulations has received considerable attention using a wide variety of techniques, see~\cite{SchumakerU,AS4r,AS3r,SuperSpline,WhiteleyComb,WhiteleyM,Homology,DimSeries,LCoho,MinReg}.  Ibrahim and Schumaker show in~\cite{SuperSpline} that the dimension of $\spl^r_d(\Delta)$, for $\Delta$ a planar triangulation and $d\ge 3r+2$, is given by a quadratic polynomial in $d$ whose coefficients are determined from simple data of the triangulation.  An important feature of planar splines is that the formula which gives the dimension of the spline space $\spl^r_d(\Delta)$ for $d\ge 3r+2$ is a lower bound for \textit{any} degree $d\ge 0$~\cite{SchumakerLower}.

In this paper we focus on splines over the union of tetrahedra all of which meet at a common vertex.  We call such a configuration a \textit{star of a vertex} (these are sometimes called \textit{cells} in the approximation theory literature~\cite{LaiSchumaker,Jimmy}).  If $\Delta$ is the star of a vertex, every spline can be written as a sum of \textit{homogeneous} splines; a homogeneous spline of degree $d$ is one which restricts to a homogeneous polynomial of degree $d$ on each tetrahedron.  We denote by $\hspl^r_d(\Delta)$ the vector space of homogeneous splines of degree $d$ in $\spl^r_d(\Delta)$.  Understanding homogeneous splines on vertex stars is crucial to computing the dimension of trivariate splines on tetrahedral complexes (see~\cite{ASWTet}) -- whose behavior even in large degree is a major open problem in numerical analysis.  We apply our present results on vertex stars to tetrahedral splines of large degree in a forthcoming paper.


In~\cite{ANS96}, Alfeld, Neamtu, and Schumaker derive formulas for the dimension of the space of homogeneous splines on vertex stars of degree $d\ge 3r+2$.  A crucial difference from the planar case is that when $d<3r+2$ these formulas may not even be a lower bound on the dimension of the space of homogeneous splines.  To explain this we differentiate between two types of vertex stars.  If the common vertex at which all tetrahedra meet is completely surrounded by tetrahedra (so that it is the unique interior vertex), then we call the vertex star a \textit{closed} vertex star.  Otherwise we call the vertex star an \textit{open} vertex star.  In Equations~15 and~16 of~\cite{ANS96}, Alfeld, Neamtu, and Schumaker define functions (in terms of simple geometric and combinatorial data of $\Delta$) which we denote by $\LBcs(\Delta,d,r)$~\eqref{eq:LBclosedstar} and $\LBos(\Delta,d,r)$~\eqref{eq:LBopenstar}, respectively.  In~\cite[Theorem~3]{ANS96}, it is shown that $\dim \hspl^r_d(\Delta)=\LBcs(\Delta,d,r)$ for $d\ge 3r+2$ if $\Delta$ is a closed vertex star and that $\dim \hspl^r_d(\Delta)=\LBos(\Delta,d,r)$ for $d\ge 3r+2$ if $\Delta$ is an open vertex star.  

\begin{remark}\label{rem:BoundCaveat}
	We have been a little imprecise -- the formulas we denote as $\LBcs(\Delta,d,r)$~\eqref{eq:LBclosedstar} and $\LBos(\Delta,d,r)$~\eqref{eq:LBopenstar} agree with Equations~15 and~16 in~\cite{ANS96}, respectively, for `most' choices of vertex coordinates of $\Delta$ but may disagree for certain choices of vertex coordinates.  See Remark~\ref{rem:EquivalenceOfFormulas}.
\end{remark}

It is straightforward to show that $\LBos(\Delta,d,r)\le \dim \hspl^r_d(\Delta)$ for all $d\ge 0$ if $\Delta$ is an open vertex star.  On the other hand, if $\Delta$ is a closed vertex star it is quite delicate to determine the degrees $d$ for which $\LBcs(\Delta,d,r)\le \hspl^r_d(\Delta)$; see~\cite{Jimmy} where a lower bound is established for homogeneous $C^2$ splines on vertex stars.  Our main result is a simple bound on the degrees $d$ for which $\LBcs(\Delta,d,r)$ is a lower bound on $\dim \hspl^r_d(\Delta)$.

\begin{theorem}[Lower bound for splines on vertex stars]\label{thm:LBGenericClosedVertexStars}
	If $\Delta$ is a closed vertex star with interior vertex $\gamma$ and $f_1^\circ$ interior edges, put
	\begin{equation}\label{eq:Dgamma}
	D_\gamma:=\left\lbrace
	\begin{array}{ll}
	2r & f^\circ_1=4\\
	\lfloor (5r+2)/3 \rfloor & f_1^\circ=5\\
	\lfloor (3r+1)/2 \rfloor & f_1^\circ\ge 6
	\end{array}.
	\right.
	\end{equation}
	If $d>D_\gamma$ then $\dim \hspl^r_d(\Delta)\ge \max\left\{\binom{d+2}{2},\LBcs(\Delta,d,r)\right\}$ and
	\[
	\dim \spl^r_d(\Delta)\ge \binom{D_\gamma+3}{3}+\sum\limits_{i=D_{\gamma+1}}^{d} \max\left\{\binom{i+2}{2},\LBcs(\Delta,i,r)\right\}.
	\]
\end{theorem}

The failure of $\LBcs(\Delta,d,r)$ to be a lower bound for $\dim \hspl^r_d(\Delta)$ in low degree is elucidated by homological techniques of Billera~\cite{Homology} as refined by Schenck and Stillman~\cite{LCoho}.  
More precisely, it follows from these techniques (in particular the Billera-Schenck-Stillman chain complex) that
\begin{equation}\label{eq:LBcsAndEulerChar}
\LBcs(\Delta,d,r)-\binom{d+2}{2}+\dim \tJ(\gamma)_d\le \dim \hspl^r_d(\Delta),
\end{equation}
where $\tJ(\gamma)$ is an ideal generated by powers of linear forms attached to the interior vertex $\gamma$ (see Proposition~\ref{prop:EulerCharBounds}).  Iarrabino showed that, via apolarity, $\dim \tJ(\gamma)_d$ can be computed from the Hilbert function of a so-called \textit{ideal of fat points} in $\mathbb{P}^2$~\cite{Iarrobino}.  The Hilbert function of an ideal of fat points in $\mathbb{P}^2$ is the subject of much research (and a major open conjecture) in algebraic geometry~\cite[Section~5]{S16}.  Fortunately, we need relatively little information about the Hilbert function of this ideal of fat points to establish Theorem~\ref{thm:LBGenericClosedVertexStars} -- a sufficiently good lower bound on the so-called Waldschmidt constant~\cite{Waldschmidt,BH102} of the dual set of points is enough to establish that $\dim \tJ(\gamma)_d=\binom{d+2}{2}$ for $d>D_\gamma$.  Evidently the inequality~\eqref{eq:LBcsAndEulerChar} then implies $\LBcs(\Delta,d,r)\le \dim\hspl^r_d(\Delta)$ if $d>D_\gamma$.

Next we turn to the question of small degree, namely when $d\le D_\gamma$.  If we use the inequality~\eqref{eq:LBcsAndEulerChar}, finding $\dim \tJ(\gamma)_d$ entails some difficult fat point computations.
However,~\eqref{eq:LBcsAndEulerChar} is often not the best possible lower bound for $\dim \hspl^r_d(\Delta)$ in small degree.  In fact, it is often easier to analyze $\hspl^r_d(\Delta)$ directly in small degree, bypassing the difficulty of computing $\dim \tJ(\gamma)_d$ entirely.  Whiteley~\cite{WhiteleyComb} completed just such an analysis for generic planar triangulations; we prove the following variation on his result for closed vertex stars.

\begin{theorem}[Low degree splines on generic closed vertex stars]\label{thm:WhitelyGenericLowDegree}
	If $\Delta$ is a generic closed vertex star with interior vertex $\gamma$ then $\dim \hspl^r_d(\Delta)=\binom{d+2}{2}$ for $d\le D_\gamma$.
\end{theorem}

\begin{remark}
	See Definition~\ref{def:GenericSimplicialComplex} for the meaning of a generic vertex star.
\end{remark}

Theorem~\ref{thm:WhitelyGenericLowDegree} shows that, at least for generic vertex positions, the best lower bound in degrees $d\le D_\gamma$ is also the simplest.  Thus, if vertex positions are generic, one cannot obtain a `better' lower bound by computing $\dim \tJ(\gamma)_d$ for $d\le D_\gamma$.  We do not know when it is possible to bypass the computation of $\dim\tJ(\gamma)_d$ in low degree for \textit{non-generic} vertex positions, although we show the general strategy for one non-generic configuration in Section~\ref{ss:nongen}.  Our result suggests that, just as in the planar case, the main difficulty in computing $\dim \hspl^r_d(\Delta)$ in low degree is understaning the non-trivial homology module of the Billera-Schenck-Stillman chain complex.

The paper is organized as follows.  In Section~\ref{sec:Background} we set up notation and briefly describe the homological machinery from~\cite{Homology,LCoho}.  In Section~\ref{sec:duality} we use apolarity and the Waldschmidt constant to show that $\dim \tJ(\gamma)_d=\binom{d+2}{2}$ if $d>D_\gamma$ (see the above discussion).  In Section~\ref{sec:bounds_cells} we prove Theorem~\ref{thm:LBGenericClosedVertexStars}, and in Section~\ref{sec:GenericLowDegree} we prove Theorem~\ref{thm:WhitelyGenericLowDegree}.  Section~\ref{sec:Examples} is devoted to illustrating our bounds in some examples and Section~\ref{sec:ConcludingRemarks} contains concluding remarks.

\section{Background and Homological Methods}\label{sec:Background}
In this section we review the necessary results from \cite{Homology} and \cite{LCoho}.  We denote by $\Delta$ a simplicial complex embedded in $\R^n$ (see~\cite{Zie} for basics on simplicial complexes).  If $n=2$ we will refer to $\Delta$ as a \emph{triangulation}, and as a \emph{tetrahedral complex} if $n=3$.
We denote by $\Delta_i^\circ$ the set of interior faces of $\Delta$ of dimension $i$, and by $f_i^\circ$ the number of such faces for $i=0,1,\ldots,n$.  If $\beta\in\Delta_i$ we call $\beta$ an $i$-face.  By an abuse of notation, we will identify $\Delta$ with its underlying space $\bigcup\limits_{\beta\in\Delta} \beta\subset\R^n$.

Recall that a simplicial complex $\Delta$ is said to be \textit{pure} if all maximal simplices have the same dimension. A pure $n$-dimensional simplicial complex $\Delta$ is \textit{hereditary} if, whenever two maximal simplices $\iota,\iota'\in\Delta_n$ intersect in a vertex $\gamma\in\Delta_0$, then there is a sequence $\iota=\iota_1,\iota_2,\ldots,\iota_k=\iota'$ of $n$-dimensional simplices satisfying that $\gamma\in\iota_i$ for $i=1,\ldots,k$ and $\iota_{i+1}\cap\iota_{i}\in\Delta_{n-1}$ for $i=1,\ldots,k-1$.

If $\Delta$ is a pure $n$-dimensional simplicial complex all of whose $n$-dimensional simplices share a common vertex $\gamma$ then we call $\Delta$ the \textit{star of a vertex} or a \textit{vertex star}. Without loss of generality, we assume that $\gamma$ is at the origin.  If $\gamma$ is an \textit{interior} vertex of $\Delta$ then we call $\Delta$ a closed vertex star; if $\gamma$ is on the boundary of $\Delta$ then we call $\Delta$ an open vertex star.  The \textit{link} of a pure $n$-dimensional vertex star in which all $n$-dimensional simplices share the vertex $\gamma$ is the set of all simplices of $\Delta$ which do not contain $\gamma$ (this has dimension $n-1$).  Throughout this article, whenever we refer to a simplicial complex $\Delta$, we will mean a pure, hereditary, $3$-dimensional vertex star whose link is simply connected.  We call these \emph{tetrahedral vertex stars}, \emph{simplicial vertex stars}, or simply \emph{vertex stars}.


We write $\tS=\R[x,y,z]$ for the polynomial ring in three variables, $\tS_d$ for the vector space of homogeneous polynomials of degree $d$, and $\tS_{\le d}$ for the vectors space of polynomials of total degree at most $d$.  For a given integer $r\geq 0$, we denote by $C^r(\Delta)$ the set of all functions $F\colon \Delta\to\R$ which are continuously differentiable of order $r$.  
\begin{definition}\label{def:SplineSpaces}
	Let $\Delta\subset\R^3$ be a tetrahedral vertex star.  The space $\spl^r(\Delta)$ of splines on $\Delta$ is the piecewise polynomial functions on $\Delta$ that are continuously differentiable up to order $r$ on $\Delta$ i.e., 
	\begin{align*}
	\spl^r(\Delta) & =  \{F\in C^r(\Delta)\colon F|_\iota\in \tS\mbox{ for all }\iota\in\Delta_3 \}.\\[5 pt]
	\intertext{If we consider polynomials of degree at most $d$, the space will be denoted by $\spl^r_d(\Delta)$, namely}
	\spl^r_d(\Delta) & = \{F\in C^r(\Delta) \colon F|_\iota\in \tS_{\le d} \mbox{ for all }\iota\in\Delta_3 \}.\\[5 pt]
	\intertext{Similarly, the space $\hspl^r_d(\Delta)$ of splines whose polynomial pieces are of degree exactly $d$ is defined as }
	\hspl^r_d(\Delta) & =  \{F\in C^r(\Delta)\colon F|_\iota\in \tS_d \mbox{ for all }\iota\in\Delta_3\}.
	\end{align*}
\end{definition}

The space $\spl^r(\Delta)$ is itself a ring, and $\tS$ includes naturally into $\spl^r(\Delta)$ as global polynomials.  In this way $\spl^r(\Delta)$ is both an $\tS$-module and a $\tS$-algebra.  We will be concerned exclusively with the structure of $\spl^r(\Delta)$ as an $\R$-vector space; however we may at times refer to the $\tS$-module structure of $\spl^r(\Delta)$.  In particular, if $\Delta$ is the star of a vertex, then it is known that
\begin{equation}\label{eq:StarVertexHomogDecomp}
\spl^r(\Delta)\cong \bigoplus\limits_{i\ge 0} \hspl^r_i(\Delta) \qquad\mbox{and}\qquad \spl^r_d(\Delta)\cong \bigoplus\limits_{i=0}^d \hspl^r_i(\Delta),
\end{equation}
where the isomorphism is as $\R$-vector spaces.  If $F\in\hspl^r_d(\Delta)$ and $G\in\tS_j$, notice that $FG\in\hspl^r_{d+j}(\Delta)$.  This means that $\spl^r_d(\Delta)$ has the structure of a \textit{graded} $\tS$-module.  

\begin{remark}
	If $\Delta$ has more than one interior vertex, there is a coning construction under which~\eqref{eq:StarVertexHomogDecomp} will still be valid.  As we focus on the case of vertex stars, we will not need this.
\end{remark}

\begin{definition}\label{def:IdealsOfFaces}
	Suppose $\Delta\subset\R^3$ is an tetrahedral vertex star.  If $\sigma\in\Delta_{2}$, let $\ell_\sigma$ be a choice of linear form vanishing on $\sigma$.  We define $\tJ(\sigma)=\langle \ell_\sigma^{r+1}\rangle$, the ideal generated by $\ell_\sigma^{r+1}$.  For any face $\beta\in\Delta_i$ where $i=0,1$ we define
	\[
	\tJ(\beta):=\sum_{\sigma\supseteq\beta} \tJ(\sigma)=\langle \ell^{r+1}_\sigma: \beta\subseteq \sigma\rangle.
	\]
	If $\beta\in\Delta_3$ we define $\tJ(\beta)=0$.
\end{definition}

\begin{proposition}{\cite[Proposition~1.2]{DimSeries}}\label{prop:AlgebraicCriterion}
	If $\Delta$ is hereditary then $F\in\spl^r(\Delta)$ if and only if
	\begin{equation*}
	F|_{\iota}-F|_{\iota'}\in \tJ(\sigma) \mbox{ for every }\iota, \iota'\in \Delta_3 \mbox{ satisfying } \iota\cap \iota'=\sigma\in\Delta_{2}.
	\end{equation*}
\end{proposition}

We define a chain complex introduced by Billera~\cite{Homology} and refined by Schenck and Stillman~\cite{LCoho}.  We refer to~\cite{Hatcher} for undefined terms from algebraic topology.  We denote the simplicial chain complex of $\Delta$ relative to its boundary $\partial\Delta$ with coefficients in $\tS$ by $\calR$:
\[
\calR \colon\quad  0\longrightarrow \tS^{f_3}\xrightarrow{\partial_3} \tS^{f^\circ_2} \xrightarrow{\partial_2} \tS^{f^\circ_1} \xrightarrow{\partial_1} \tS^{f^\circ_0} \longrightarrow 0.
\]
The ideals $\tJ(\beta)$ fit together to make a sub-chain complex of $\calR$ (the differential is the restriction of the differential of $\calR$):
\[
\calJ\colon \quad  0\longrightarrow  \bigoplus_{\sigma\in\Delta^\circ_{2}} \tJ(\sigma)\xrightarrow{\partial_{2}} \bigoplus_{\tau\in\Delta^\circ_1} \tJ(\tau) \xrightarrow{\partial_1} \bigoplus_{\gamma\in\Delta^\circ_0} \tJ(\gamma) \longrightarrow 0.
\]
The Billera-Schenck-Stillman chain complex is the quotient of $\calR$ by $\calJ$, namely
\[
\calR/\calJ \colon \quad 0\longrightarrow \bigoplus_{\iota\in\Delta_{3}}  \tS\xrightarrow{\overline{\partial}_3} \bigoplus_{\sigma\in\Delta_{2}^\circ} \frac{\tS}{\tJ(\sigma)}\xrightarrow{\overline{\partial}_{2}} \bigoplus_{\tau\in\Delta_1^\circ}\frac{\tS}{\tJ(\tau)}\xrightarrow{\overline{\partial}_1} \bigoplus_{\gamma\in\Delta_0^\circ}\frac{\tS}{\tJ(\gamma)}\longrightarrow 0\ .
\]
These three chain complexes fit into the evident short exact sequence of chain complexes
\[
0\longrightarrow \calJ \longrightarrow \calR \longrightarrow \calR/\calJ \longrightarrow 0.
\]
As is standard notation, $\calR_i,\calJ_i,$ and $(\calR/\calJ)_i$ refer to the modules in the chain complex at homological position $i$.  For instance, $\calR_0=\tS^{f^\circ_0},\calR_1=\tS^{f^\circ_1}$, and so on.  We summarize some well-known properties of $\calR/\calJ$ (see~\cite{LCoho,Spect}).

\begin{proposition}\label{prop:FrequentlyUsedIsomorphisms}
	If $\Delta\subset\R^3$ is a tetrahedral vertex star whose link is simply connected, then $\spl^r(\Delta)\cong H_3(\calR/\calJ)\cong \tS\oplus H_{2}(\calJ)$, $H_i(\calR/\calJ)\cong H_{i-1}(\calJ)$ for $i=1,2$, and $H_0(\calR/\calJ)=0$.
\end{proposition}

The inclusion of $\tS$ into $\spl^r(\Delta)$ as globally polynomial corresponds (via the isomorphism in Proposition~\ref{prop:FrequentlyUsedIsomorphisms}) to the copy of $\tS$ in $\tS\oplus H_{2}(\calJ)$, while the map
\[
\bigoplus_{\sigma\in\Delta^\circ_{2}} \tJ(\sigma) \xrightarrow{\partial_{2}} \bigoplus_{\tau\in\Delta^\circ_{1}} \tJ(\beta)
\]
encodes the so-called \textit{smoothing cofactors}.  By Proposition~\ref{prop:FrequentlyUsedIsomorphisms}, the Billera-Schenck-Stillman chain complex $\calR/\calJ$ and the chain complex $\calJ$ contain essentially the same information.

We now put everything together to write the dimension of $\hspl^r_d(\Delta)$ in terms of the Billera-Schenck-Stillman chain complex.  If $\mathcal{C}=0\rightarrow C_n\rightarrow\cdots\rightarrow C_0\rightarrow 0$ is a chain complex of graded $\tS$-modules, we write $\chi(\mathcal{C},d)$ for the graded Euler-Poincar\'{e} characteristic of $\mathcal{C}$.  That is,
\[
\chi(\mathcal{C},d)=\sum_{i=0}^n(-1)^{n-i}\dim(C_i)_d.
\]

\begin{proposition}\label{prop:EulerCharacteristicAndDimension}
	If $\Delta$ is a tetrahedral vertex star then
	\begin{equation}\label{eq:celldimformula}
	\dim \hspl^r_d(\Delta)=\chi(\calR/\calJ,d)+\dim H_2(\calR/\calJ)_d=\dim \tS_d+\chi(\calJ,d)+\dim H_1(\calJ)_d.
	\end{equation}
	In particular, 
	\[
	\dim \hspl^r_d(\Delta)\ge\chi(\calR/\calJ,d)=\dim\tS_d+\chi(\calJ,d).
	\]
\end{proposition}
\begin{proof}
	We use the fact that $\chi(\calJ,d)=\sum_{i=0}^3 (-1)^{3-i} \dim H_i(\calJ)_d$.	 If $\Delta$ is a closed vertex star with interior vertex $\gamma$, then $\calJ$ has the form
	\[
	\calJ\colon \quad 0\rightarrow \bigoplus_{\sigma\in\Delta^\circ_2} \tJ(\sigma)\xrightarrow{\partial_2} \bigoplus_{\tau\in\Delta^\circ_1} \tJ(\tau)\xrightarrow{\partial_1} \tJ(\gamma)\rightarrow 0,
	\]
	the map $\partial_1$ is surjective from the definition of $\tJ(\gamma)$, hence $H_0(\calJ)=0$.  Thus
	\[
	\dim H_2(\calJ)_d-\dim H_1(\calJ)_d=\chi(\calJ,d).
	\]
	The result follows since $\dim \hspl^r_d(\Delta)=\dim\tS_d+\dim H_2(\calJ)_d$ by Proposition~\ref{prop:FrequentlyUsedIsomorphisms}.
	
	If $\Delta$ is an open vertex star, then $\calJ$ has the form
	\[
	\calJ\colon \quad 0\rightarrow \bigoplus_{\sigma\in\Delta^\circ_2} \tJ(\sigma)\xrightarrow{\partial_2} \bigoplus_{\tau\in\Delta^\circ_1} \tJ(\tau)\xrightarrow{\partial_1} 0\rightarrow 0,
	\]
	so there is not even a vector space in homological index $0$, thus $H_0(\calJ)=0$ as well and the formula follows from the above argument immediately.
\end{proof}

Finally, we clarify what we mean by \textit{generic} vertex positions for a tetrahedral vertex star; this is fairly standard in the literature on splines~\cite{WhiteleyM,WhiteleyComb,Homology,ASWTet,ANS96}.  The main point is that it will suffice to prove Theorem~\ref{thm:LBGenericClosedVertexStars} when $\Delta$ is a generic tetrahedral vertex star.

\begin{definition}\label{def:GenericSimplicialComplex}
	Suppose $\Delta\subset\R^3$ is a star of the vertex $\gamma$.  We call $\Delta$ \textit{generic} with respect to a fixed $r,d$ if, for all sufficiently small perturbations of the vertices $\gamma'\neq\gamma\in\Delta_0$, the resulting vertex star $\Delta'$ satisfies $\dim \hspl^r_d(\Delta')=\dim \hspl^r_d(\Delta)$.  If $r$ and $d$ are understood from context, then we simply say $\Delta$ is generic.
\end{definition}

\begin{lemma}\label{lem:Generic}
	Suppose $\Delta\subset\R^3$ is a star of the vertex $\gamma$, and fix non-negative integers $r$ and $d$.  Then almost all sufficiently small perturbations of the vertices $\gamma'\neq \gamma\in\Delta_0$ result in a vertex star $\Delta'$ which is generic with respect to $r$ and $d$.  Moreover $\dim \hspl^r_d(\Delta)\ge \dim \hspl^r_d(\Delta')$.
\end{lemma}
\begin{proof}
	This follows immediately from examining rank conditions on any of the equivalent ways of defining splines as the kernel of a linear transformation.
\end{proof}

\section{Duality: fat points and powers of linear forms}\label{sec:duality}
In this section we review a duality between ideals generated by powers of linear forms and ideals of polynomials which vanish to certain orders on sets of points in projective space, called \textit{fat point ideals}.  We reduce the presentation to our case i.e., ideals in the polynomial ring of three variables and the corresponding fat points ideals in $\Proj^2$.  We use this duality, along with combinatorial bounds from~\cite{CHT11}, to prove Corollary~\ref{cor:regularity}, the main result of this section.  Corollary~\ref{cor:regularity} provides explicit lower bounds for the degree in which $\tJ(\gamma)_d=\tS_d$, where $\gamma$ is the interior vertex of a closed vertex star.

We write $[a:b:c]$ for a point in projective $2$-space over $\R$, which we denote by $\Proj^2$.  We let $\tR:=\R[X,Y,Z]$ be the polynomial ring in three variables.  If $P=[a:b:c]\in\Proj^2$ we write $\wp_P$ for the ideal of homogeneous polynomials in $\tR$ which vanish at $P$; i.e. $\wp_P=\langle bX-aY,cX-aZ,cY-bZ\rangle$.  It is straightforward to see that $\wp_P^m$ consists of all polynomials whose homogeneous components vanish to order $m$ at $P$.

\begin{definition}\label{def:FatPointIdeal}
	Let $\X=\{P_1,\ldots,P_k\}$ be a collection of points in $\Proj^2$ and ${\bf m}=\{m_1,\ldots,m_k\}$ a collection of positive integers attached to $P_1,\ldots,P_k$, respectively.  The \emph{ideal of fat points} associated to $\X$ and $\bf m$ is
	\[
	\tI=\tI_\X^{\bf m}:=\bigcap_{i=1}^k \wp_i^{m_i}.
	\]
	If there is a positive integer $s$ so that $m_i=s$ for $i=1,\ldots,k$ then we write $\tI^{(s)}_\X$ instead of $\tI^{\bf m}_\X$.  If $m_i=1$ for $i=1,\ldots,k$ we simply write $\tI_\X$.
\end{definition}

\begin{remark}
	It is straightforward to see that $\tI^{\bf m}_\X$ is the set of polynomials whose homogeneous components vanish to order $m_i$ at the point $P_i$, for $i=1,\ldots,k$.  Since $\wp_i^{m_i}$ is graded for each $\sigma\in\Omega$, $\tI^{\bf m}_\X$ is also a graded ideal.
\end{remark}

\begin{remark}
	The ideal $I^{(s)}_\X$ in Definition~\ref{def:FatPointIdeal} is called the $s$th \textit{symbolic power} of $I_{\X}$, and consists of the polynomials whose homogeneous components vanish to order $s$ on $\X$.  The $s$th symbolic power $\tI^{(s)}$ can be defined for any ideal $\tI\subset\tS$, but the definition given above for points is all we will need.
\end{remark}


Now consider simultaneously the polynomial rings $\tS=\R[x,y,z]$ and $\tR=\R[X,Y,Z]$, and let $\tR$ act on $\tS$ as polynomial differential operators.  Namely, if $h\in \tR$ and  $f\in \tS$ then $h\circ f = h\bigl(\frac{\partial}{\partial x}, \frac{\partial}{\partial y}, \frac{\partial }{\partial z}\bigr)\circ f$.  We call this the \emph{apolarity action} of $\tR$ on $\tS$.  The apolarity action induces a perfect pairing $\tR_d\times \tS_d\to \R$ by $(h,f)\to h\circ f$.

\begin{example}
	Let $F=X^2+Y^2+Z^2\in\tR$.  If $f\in\tS$, then $F\circ f=\frac{\partial^2 f}{\partial x^2}+\frac{\partial^2 f}{\partial y^2}+\frac{\partial^2 f}{\partial z^2}$.
\end{example}

If $\tI\subset \tR$ is an ideal of $\tR$, then the \emph{inverse system} $\inverseSystem{\tI}$ of $\tI$ is defined as
\[
\inverseSystem{\tI}:=\{f\in\tS: h\circ f=0 \mbox{ for all } h\in \tI\}.
\]
If $\tI$ is graded, then $\inverseSystem{\tI}$ is a graded vector space (it is generally \textit{not} an ideal) with graded structure $\inverseSystem{\tI}\cong\bigoplus_{d\ge 0} \inverseSystem{\tI}_d$, where $\inverseSystem{\tI}_d$ is the vector space of all homogeneous polynomials of degree $d$ in $\inverseSystem{\tI}$.

\begin{example}
	Let $P=[0:0:1]\in\Proj^2$, with $\wp_P=\langle X,Y\rangle$.  Then $\inverseSystem{\wp_P}=\spn\{1,z,z^2,z^3,\ldots\}$.  More generally, if $P=[a:b:c]$ then $\inverseSystem{\wp_P}=\spn\{1,ax+by+cz,(ax+by+cz)^2,\ldots\}$.
\end{example}

For a graded ideal $\tI\subset\tR$, the apolarity action induces an isomorphism of vector spaces $(\tR/\tI)_d\cong \inverseSystem{\tI}_d$ (this follows since the apolarity action induces a perfect pairing $\tR_d\times \tS_d\to \R$).  Thus one can deduce $\dim \inverseSystem{\tI}_d$ from $\dim \tI_d$, and vice-versa.  The following result of Iarrobino~\cite{Iarrobino} describes the inverse system of a fat point ideal.

We first introduce some notation which suits our context.  If $\sigma$ is a two-simplex in $\R^3$ whose affine span contains the origin, let $\ell_\sigma=ax+by+cz$ be a choice of linear form vanishing on $\sigma$ (well-defined up to constant multiple).  The coefficients of $\ell_\sigma$ define the point $P_\sigma= [a:b:c]\in\Proj^2$ (notice this point does not depend on the multiple of $\ell_\sigma$ chosen), which in turn defines the ideal $\wp_\sigma=\wp_{P_\sigma} \subseteq \tR$.

%
%

\begin{theorem}[Iarrabino~\cite{Iarrobino}]\label{theor:iarrobino}
	Let $\Omega$ be a collection of $2$-simplices in $\R^3$ each of whose affine span contains the origin and let $m_\sigma$ be a positive integer attached to each $\sigma\in\Omega$.  Put $\X=\{P_\sigma\}_{\sigma\in\Omega}$.  If $d\ge \max\{m_\sigma+1\}$, let $d-\mathbf{m}=\{d-m_\sigma\}_{\sigma\in\Omega}$ and $\tI=\tI^{d-\mathbf{m}}_\X$.  Then
	\begin{equation*}
	\bigl\langle \ell_{\sigma}^{m_\sigma+1}\colon \sigma\in \Omega\bigr\rangle_d =
	%
	\begin{cases}
	0 &\mbox{for\; } d\leq \max\{m_\sigma\} \\
	(\inverseSystem{\tI})_d\cong \left(\dfrac{\tR}{\tI}\right)_d &\mbox{for\; } d\geq \max\{m_\sigma+1\}
	\ . 
	\end{cases}
	\end{equation*} 
\end{theorem}

Theorem~\ref{theor:iarrobino} has an especially nice formulation in the case of uniform powers.  We state this for the ideal $\tJ(\gamma)$ of the interior vertex $\gamma$ of a closed vertex star, as this is the case of interest to us.

\begin{corollary}\label{cor:uniformiarrobino}
	Suppose $\Delta\subset\R^3$ is a vertex star with unique interior vertex $\gamma$, so $\tJ(\gamma)=\langle \ell_\sigma^{r+1}: \sigma\in\Delta^\circ_2\rangle$.  Put $\X=\{P_\sigma\}_{\sigma\in\Delta^\circ_2}$ and let $\tI_\X=\cap_{\sigma\in\Delta^\circ_0} \wp_\sigma\subset\tR$.  Then
	\[
	\dim \tJ(\gamma)_d=\left\lbrace
	\begin{array}{ll}
	0 & d \le r\\
	\dim\left(\dfrac{\tS}{\tI_\X^{(d-r)}}\right)_d & d\ge r+1
	\end{array}
	\right.
	\]
\end{corollary}

The proof of Theorem \ref{theor:iarrobino} can be found in  \cite{Iarrobino}, see also \cite{Geramita} for an introduction to inverse system of fat points and \cite{FatPoints} for the connection between fat points and splines.

\begin{example}\label{Ex:regOct}
	Let $\Delta$ be the regular octahedron with central vertex $\gamma$ at the origin and vertices at $(\pm 1, 0, 0), (0,\pm 1, 0)$ and $(0,0,\pm 1)$. 
	Then there are 12 interior two-dimensional faces 
	which we denoted as $\sigma_1,\ldots,\sigma_{12}$; we number them so that they lie in the planes defined by the linear forms $\ell_{\sigma_i}=x$, for $i=1,\dots, 4$, $\ell_{\sigma_i}=y$ for $i=5,\dots 8$, and $\ell_{\sigma_i}=z$ for $i= 9,\dots, 12$. The dual points are $P_{\sigma_1}=[1: 0: 0]$, $P_{\sigma_5}= [0:1:0]$, and $P_{\sigma_9}=[0:0:1]$, see the graph on the left of Figure \ref{Fig:GenOcta}. 
	These points define the ideals $\wp_{\sigma_1}=\langle Y,Z\rangle$, $\wp_{\sigma_5}=\langle X,Z\rangle$ and $\wp_{\sigma_9}=\langle X,Y\rangle$.
	For a positive integer $r$, and $d\geq r+1$ let $\tI= \cap_{\sigma\in\Delta_2^\circ}\wp_{\sigma}^{d-r}$. 
	Theorem \ref{theor:iarrobino} says that 
	\[\dim \tJ(\gamma)_d= \dim \bigl\langle \ell_{\sigma}^{r+1}\colon \sigma\in\Delta_2^\circ\bigr\rangle_d= \dim (\tR/I)_d. \]
	For example, if $r=1$ and $d=3$ then 
	$I= \langle X, Y\rangle^2  \cap \langle X,Z\rangle^2 \cap \langle Y,Z\rangle^2$ and $\dim (\tR/I)_3 = 9$. On the other hand, $\tJ(\gamma) = \langle x^2, y^2, z^2\rangle $ and $\dim \tJ(\gamma)_3 = 9$. 
	
\end{example}

\begin{example}\label{Ex:genOct}
	Let now $\Delta'$ be a generic octahedron with central vertex $\gamma$ at the origin. Then the 12 two-dimensional faces $\sigma_i\in \Delta_2'$ lie on 12 different planes through the origin of $\R^3$ defined by the linear forms $\ell_i=a_ix+b_iy+c_iz$. These linear forms define 12 points $[a_i:b_i:c_i]$ in $\Proj^2$. Notice that each of the edges $\tau_j\in{\Delta'}_1^\circ$ lies in the intersection of four of these planes, in $\Proj^2$ that means that the four dual points to the linear forms vanishing at those four planes lie on a line -- thus there is a dotted line in Figure~\ref{Fig:GenOcta} for every interior edge of $\Delta'$. The dual diagram in $\Proj^2$ is illustrated on the right in Figure \ref{Fig:GenOcta}.
\end{example}

\begin{figure}[htbp]
	\centering
	\includegraphics[height=4.5cm]{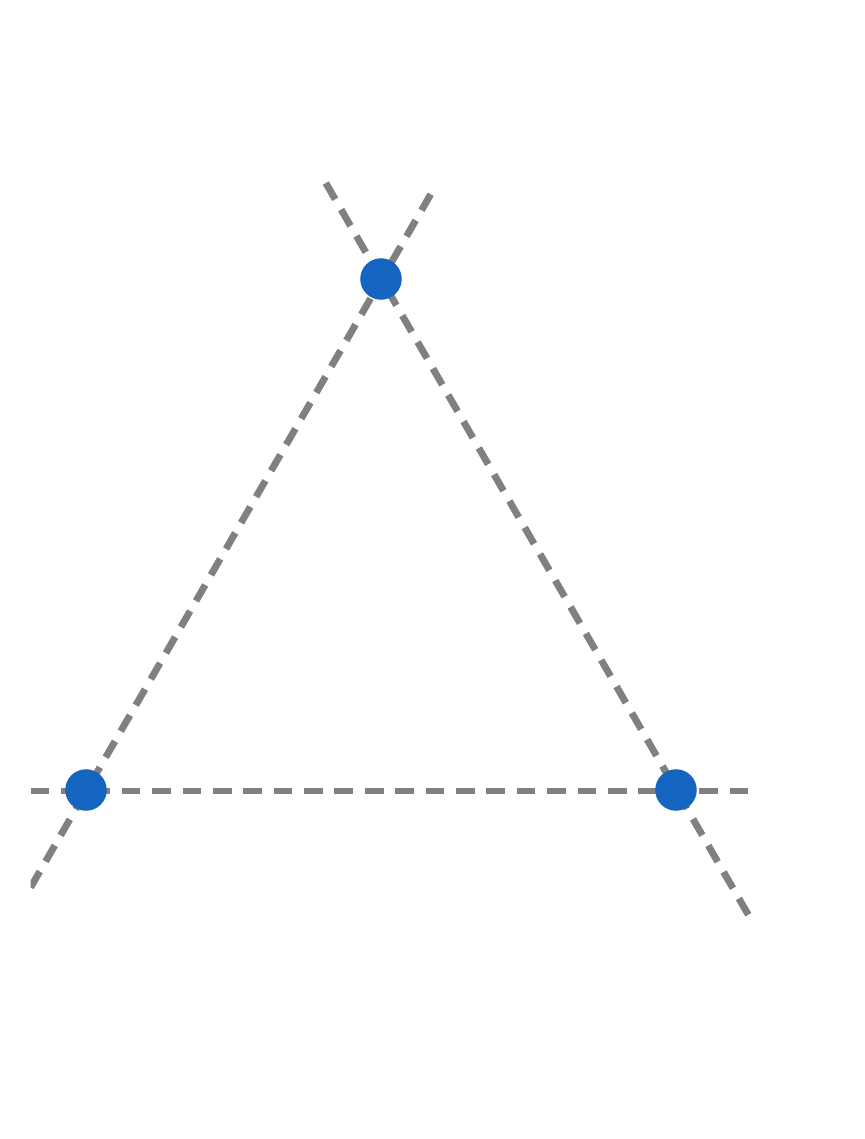}\qquad	\includegraphics[height=4.5cm]{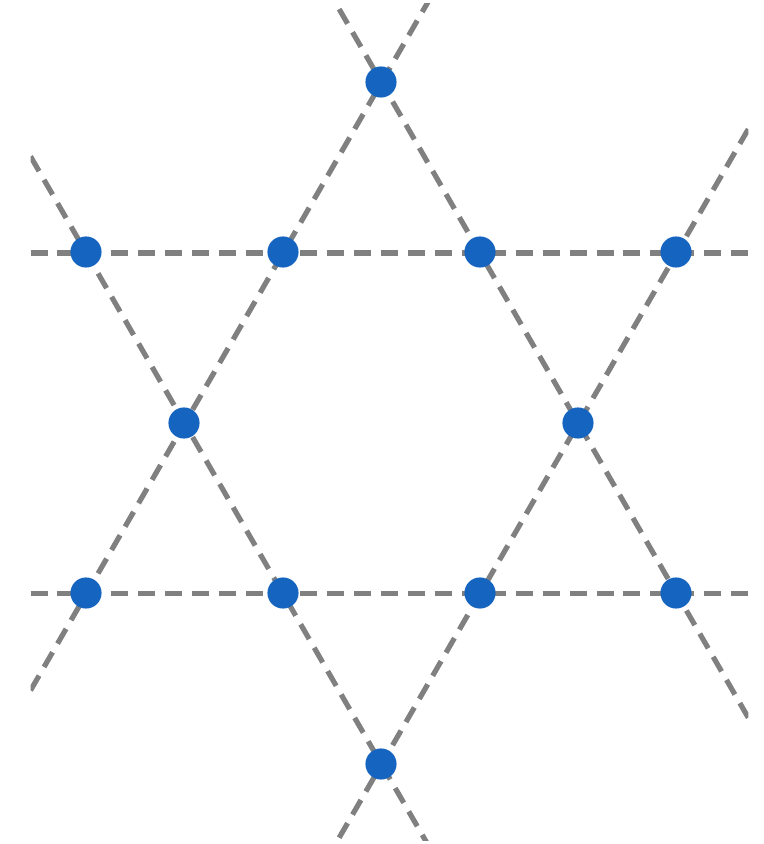}
	\caption{Dual graph of a regular octahedron $\Delta$ (on the left) and a generic octahedron $\Delta'$ (on the right).}\label{Fig:GenOcta}
\end{figure}


\subsection{The Waldschmidt constant}
Given a graded ideal $\tI\subset \tR$, we put $\alpha(\tI):=\min\{d:\tI_d\neq 0\}$.  For instance, if $\X$ is a set of points in $\Proj^2$, then $\alpha(\tI_\X)$ is the minimum degree of a homogeneous polynomial which vanishes on $\X$.  An asymptotic invariant attached to the ideal $\tI_\X$ which has been studied in many different contexts is the Waldschmidt constant, defined as
\[
\walpha(\tI_\X)=\inf_{s> 0}\left\lbrace\frac{\alpha(\tI_\X^{(s)})}{s}\right\rbrace
\]
It is known that the Waldschmidt constant is actually a limit (this follows from subbaditivity of the sequence $\alpha(\tI_\X^{(s)})$ - see~\cite[Lemma~2.3.1]{BH102}); so $\walpha(I)=\lim\limits_{s\to\infty} \frac{\alpha(\tI_\X^{(s)})}{s}$.

\begin{remark}
	The limit $\lim_{s\to\infty}\alpha(\tI_\X^{(s)})/s$ was first introduced by Waldschmidt~\cite{Waldschmidt} in complex analysis, although the ideas behind the Waldschmidt constant go back at least to Nagata's solution to Hilbert's fourteenth problem.  In commutative algebra, the Waldschmidt constant gives bounds related to the \textit{containment problem}; in other words for what pairs of integers $(r,s)$ we have the containment $\tI^{(s)}\subset \tI^r$ for an ideal $\tI$ in a polynomial ring (the first use of the Waldschmidt constant for these purposes is in~\cite{BH102}).
\end{remark}

\begin{proposition}\label{prop:regularitybound}
	For a closed vertex star $\Delta\subset\R^3$, let $\Omega\subset\Delta^\circ_2$ be a finite subset of two-faces such that $\dim\spn\{\ell_{\sigma}\}_{\sigma\in\Omega}=3$.  Put $\tJ(\gamma)=\langle\ell_\sigma^{r+1}:\sigma\in\Omega\rangle$, $\X=\{P_\sigma\}_{\sigma\in\Omega}$, and $\tI_\X=\cap_{\sigma\in\Omega}\wp_\sigma$.  Then $\tJ(\gamma)_d=\tS_d$ for $d> \dfrac{\walpha(\tI_\X)r}{\walpha(\tI_\X)-1}$.  
\end{proposition}

\begin{proof}
	By Corollary \ref{cor:uniformiarrobino}, $\tJ(\gamma)_d=\tS_d$ if and only if $d< \alpha(\tI_\X^{(d-r)})$.  We may assume $d> r$ (otherwise $\tJ(\gamma)_d=0$).  Dividing both sides by $d-r$ gives $\tJ(\gamma)_d=\tS_d$ if and only if
	\[
	\frac{d}{d-r}< \frac{\alpha(\tI_\X^{(d-r)})}{d-r}.
	\]
	Since the right hand side is larger than $\widehat{\alpha}(\tI_\X)$, we see that 
	\[
	\mbox{if }\frac{d}{d-r}< \widehat{\alpha}(\tI_\X)\mbox{ then } \tJ(\gamma)_d=\tS_d.
	\]
	Solving for $d$ yields the proposition, provided that $\widehat{\alpha}(\tI_\X)>1$.  This latter inequality follows from a result of Chudnovsky that $\walpha(\tI_\X)\ge \frac{\alpha(\tI_\X)+1}{2}$ (see~\cite[Proposition~3.1]{HH13}).  Thus if $\alpha(\tI_\X)\ge 2$ then $\widehat{\alpha}(\tI_\X)\ge \frac{3}{2}>1$, so $\widehat{\alpha}(\tI_\X)=1$ if (and only if) $\alpha(\tI_\X)=1$, that is $\X=\{P_\sigma\}_{\sigma\in\Omega}$ is contained in a line.  But this would imply that the span of the corresponding linear forms $\{\ell_\sigma\}_{\sigma\in\Omega}$ is at most two dimensional, contrary to assumption.
\end{proof}

\subsection{A reduction procedure for fat points}\label{subsec:reductionp}
Following the notation introduced in Section \ref{sec:duality}, let $\Omega$ denote a collection of faces in $\Delta_2^\circ$ of a vertex star $\Delta$. The dual points defined by the linear forms vanishing on these faces define the dual points $\{P_\sigma\}_{\sigma\in\Omega}\subseteq \Proj^2$. Consider a collection $\mathbf{m}$ of non-negative integers $\{m_\sigma\}_{\sigma\in\Omega}$, and the fat points ideal $I_\Omega^{\mathbf{m}}=\cap_{\sigma\in\Omega}\wp_\sigma^{m_\sigma}$. 

If $\tau\in\Delta_1^\circ$, then $\tau$ is the intersection of at two (distinct) planes in $\R^3$ which contain the faces $\Omega_\tau\subseteq \Omega$ having $\tau$ as one of their edges. Since $\tau\in\cap_{\sigma\in\Omega_\tau}\sigma$, then the dual points $\{P_{\sigma}\colon \sigma\in\Omega_\tau \}$ lie in a common line $L_\tau$ in $\Proj^2$. By construction, for each interior edge $\tau$, the corresponding dual line $L_\tau$ contains at least two points $P_\sigma$ for $\sigma\in\Delta_2^\circ$. 

In the following we describe the procedure introduced by Cooper, Harbourne, and Tietler in \cite{CHT11} to give bounds on $\dim (I_\Omega^{\mathbf{m}})_d$.  This is done by constructing the so-called \emph{reduction vector}, which we now describe.  Given a sequence of non-negative integers $\mathbf{m} $, a collection of points $\{P_\sigma\}_{\sigma\in\Omega}$, and the sequence of lines $L_1,\dots, L_n$ of not necessarily different lines from the collection $\{L_\tau\}_{\tau\in\Delta_1^\circ}$, the vector $\mathbf{d}=(d_1,\dots, d_n)$ is defined  inductively as follows. 

\begin{enumerate}
	\item Starting with $L_1$, we define $d_1$ as the number 
	of points lying on $L_1$, counted with multiplicity. Namely, if $L_1=L_\tau$ for some $\tau\in\Delta_1^\circ$, then $d_1= \sum_{\sigma\in\Omega_\tau}m_\sigma$.\label{Reduction-step1}
	\item Reduce by 1 the multiplicities of all the points lying on $L_1$ and consider the sequence of points $\{P_\sigma\}_{\sigma\in\Omega}$ now with multiplicities $m_\sigma$ for $\sigma\notni \tau$, and $m_\sigma-1$ for $\sigma\ni \tau$. \label{Reduction-step2}
	\item Repeat \eqref{Reduction-step1} for $L_i$ for $i=2,\dots, n$ and the sequence of points with reduced multiplicities obtained in \eqref{Reduction-step2} at step $i-1$. 
\end{enumerate}

A reduction vector $\mathbf{d} =(d_1,\dots,d_n)$ is said to be a \emph{full reduction vector} for the fat points ideal $I_{\Omega}^{\mathbf{m}}=\cap_{\sigma\in\Omega}\wp_{\sigma}^{m_\sigma}$ if $\sum_{i=1}^n d_i=\sum_{\sigma\in\Omega} m_\sigma$.

\begin{example}
	Let $\Delta$ be the regular octahedron from Example \ref{Ex:regOct}. Taking $m_\sigma= 2$ for every $\sigma\in \Delta_2^\circ$, and the ideal of fat points 
	$I= \wp_{\sigma_1}^{2}\cap \wp_{\sigma_5}^2\cap \wp_{\sigma_9}^2$.
	The set of lines in this case is $\{L_\tau\colon \tau\in\Delta_1^\circ\}= \{X, Y, Z\}$. If we take $L_1=Y$, two of the points lie on $L_1$, each of them with multiplicity 2, so $d_1= 4$. Notice that $\langle I, Y\rangle  = \langle X^2Z^2, Y\rangle$. See Figure \ref{Fig:regOct}, where we produce a reduction vector following starting from the dual graph of $\Delta$ in $\Proj^2$.
	
	\begin{figure}[htbp]
		\centering
		\includegraphics[height=2.8cm]{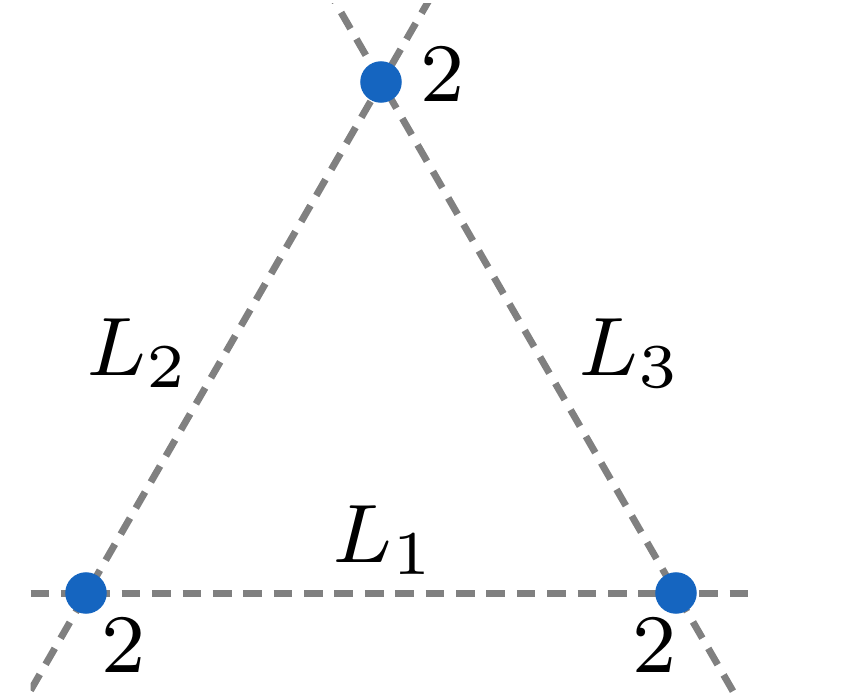}	\includegraphics[height=2.8cm]{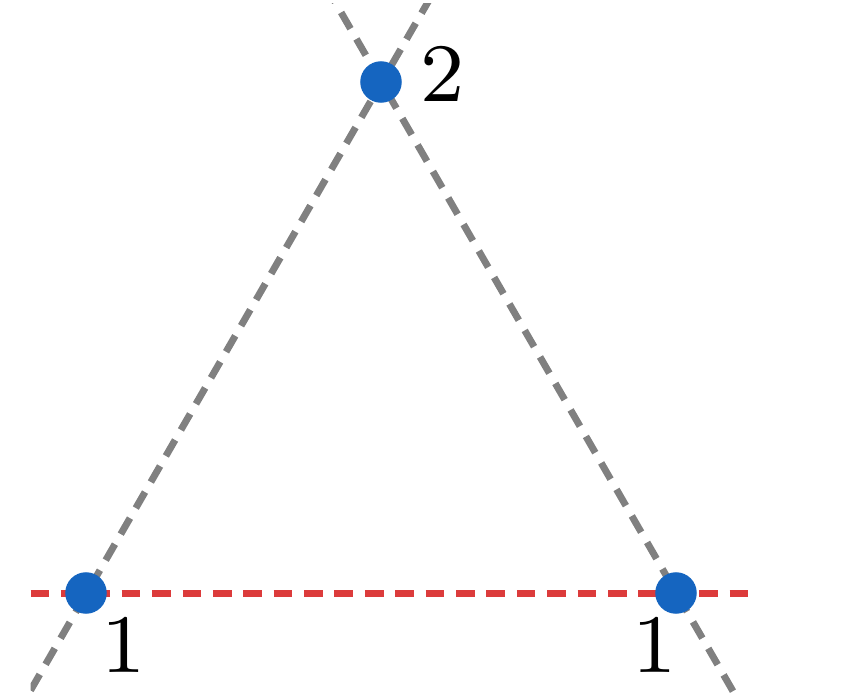}
		\includegraphics[height=2.8cm]{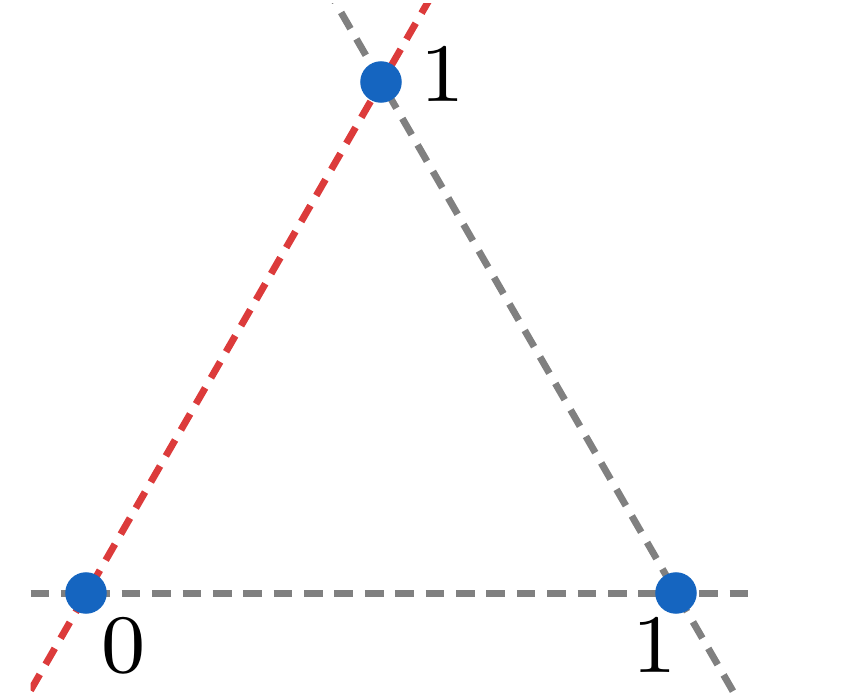}
		\includegraphics[height=2.8cm]{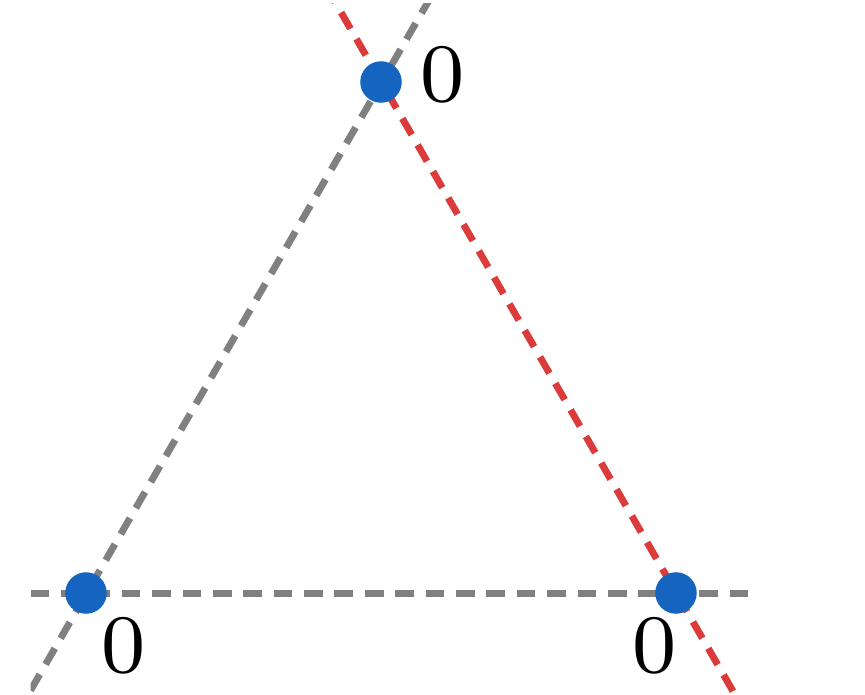}
		\caption{The graph of the regular octahedron $\Delta$ is composed of three points, we consider each of them to have multiplicity 2. The interior edges of $\Delta$ lie in three different lines, each of them correspond to a line in $\Proj^2$ as illustrated in the first graph on the left. Taking the sequence $L_1,L_2,L_3$, the reduction consists of the three steps (from left to right). Notice that the multiplicities are reduced to zero, and the  reduction vector is $\mathbf{d} = (4,3, 2)$. }\label{Fig:regOct}
	\end{figure}
\end{example}

\begin{example}\label{Ex:genOctDual}
	Let $\Delta'$ be the generic octahedron from Example \ref{Ex:genOct}. Let us take $m_\sigma= 4$ for every $\sigma\in {\Delta'}_2^\circ$, and the ideal of fat points given by 
	$I= \cap_{i=1}^{12}\wp_{\sigma_i}^{4}$.
	To construct a reduction vector associated to the ideal $I$, we can take any sequence of lines in $\{L_\tau\colon \tau\in\Delta_1^\circ\}$, in particular we can take a sequence so that all multiplicities reduce to zero. For instance, following the notation in Figure \ref{Fig:genOct}, by taking the sequence of lines $L_1,L_6,L_4,L_5, L_3, L_2$ the multiplicity at each point is reduced to 2. If we continue the reduction following the sequence of lines in the same order one more time, we get $\mathbf{d} = (16,16,14, 14, 12, 12,8,8, 6,6,4,4, 2, 2)$, and all the multiplicities are reduced to zero.
	\begin{figure}[htbp]
		\centering
		\includegraphics[height=4.5cm]{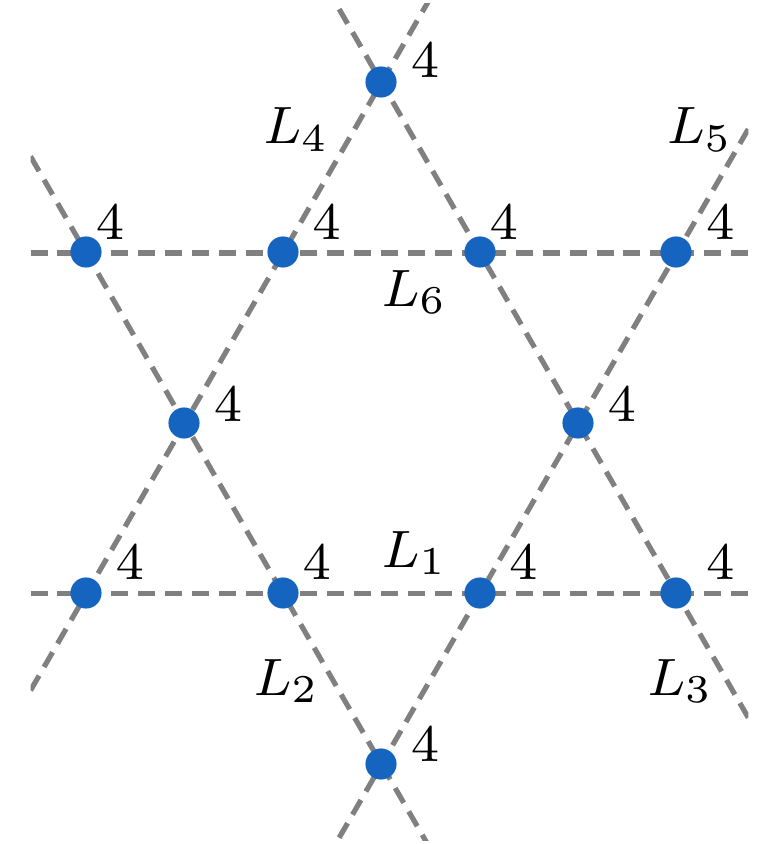}\quad	\includegraphics[height=4.5cm]{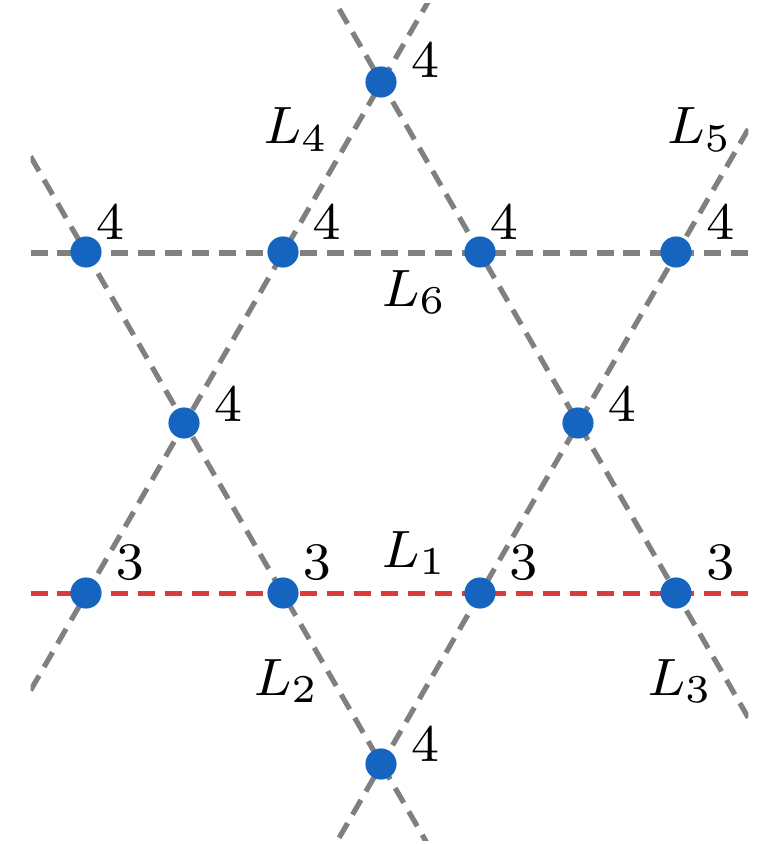}\quad
		\includegraphics[height=4.5cm]{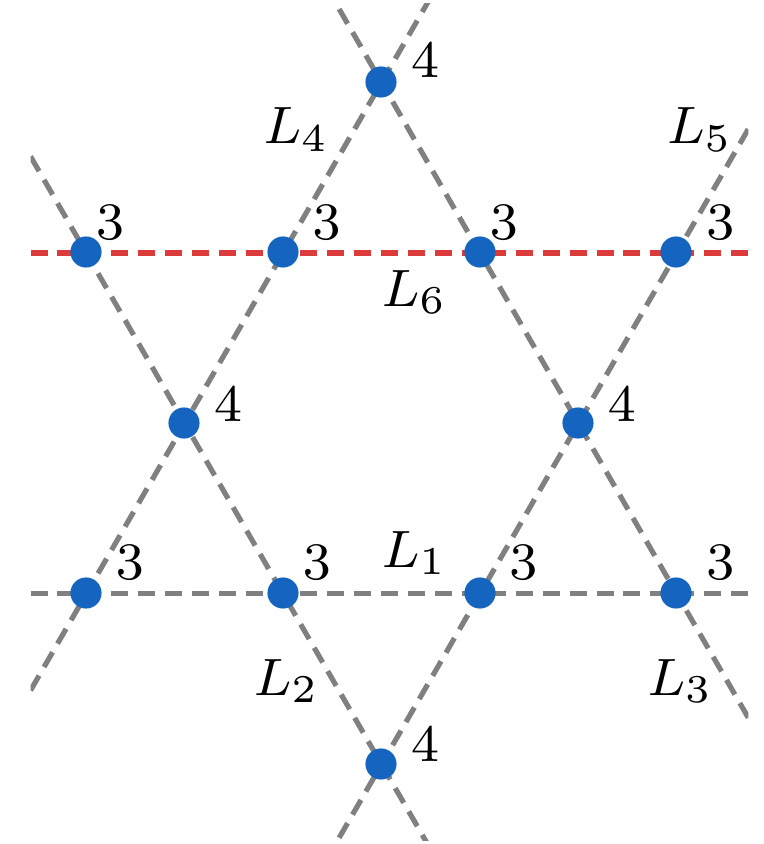}
		\medskip		
		\includegraphics[height=4.5cm]{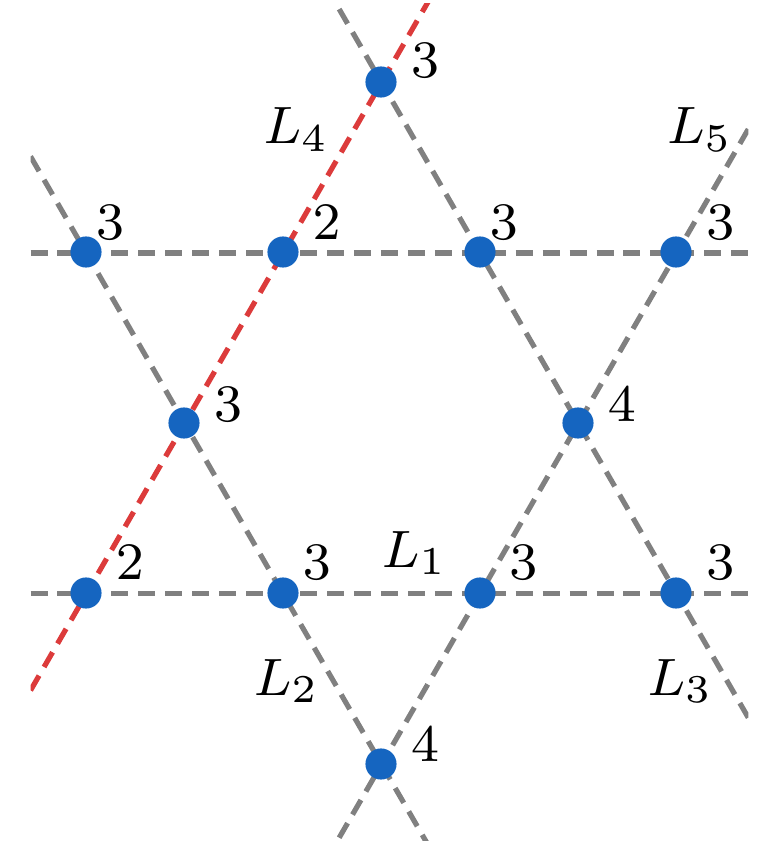}\quad
		\includegraphics[height=4.5cm]{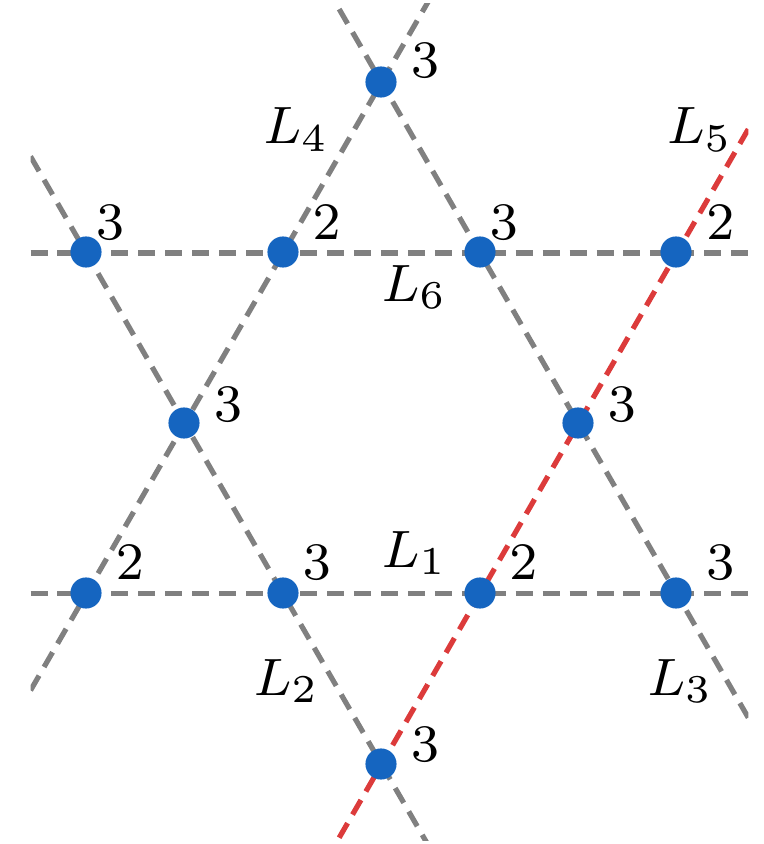}
		\caption{The graph of a generic octahedron $\Delta'$ is composed by 12 points, we consider each of them to have multiplicity 4. The interior edges of $\Delta'$ lie on 6 different lines, each of them correspond to a line in $\Proj^2$ as illustrated in the first graph on the left. Taking the sequence $L_1,L_6,L_4, L_5$, the reduction consists of the four steps (from left to right). The  reduction vector for this sequence of lines is $\mathbf{d} = (16,16, 14, 14)$. }\label{Fig:genOct}
	\end{figure}
\end{example}	

In \cite{CHT11} it is shown that the reduction vector $\mathbf{d}$ yields bounds on $\dim\bigl(\tI_{\X}^{\mathbf{m}}\bigr)_d$.  In the statement of the following theorem (and throughout this document) it is important that we use the convention that $\binom{a}{b}=0$ if $a<b$.

\begin{theorem}\cite[Corollary~2.1.5]{CHT11}\label{thm:CHT11}
	Let $\tI^{\bf m}_\X=\cap_{\sigma\in\Omega} \wp_\sigma^{m_\sigma}$ be a fat points ideal and $\mathbf{d}=(d_1,\ldots,d_n)$ a full reduction vector from the sequence of lines $L_1,\ldots,L_n$.  Let $h'_{n}=\binom{d-n+2}{2}$ and for $0\le i<n$ let
	\[
	h'_i=\binom{d-i+2}{2}-\sum\limits_{i+1\le j\le n} d_j.
	\]
	Then
	\[
	\max\bigl\{h'_0,\ldots,h'_n\bigr\}\le \dim \bigl(\tI^{\bf m}_\X\bigr)_d\le \binom{d-n+2}{2}+\sum\limits_{i=0}^{n-1}\binom{d-i-d_{i+1}+1}{1}\ .
	\]
\end{theorem}
%
%
If the reduction vector does not contain zeros (it is \textit{positive}), the following corollary to Theorem~\ref{thm:CHT11} provides bound on the initial degree of the ideal of fat points $I_{\Omega}^{\mathbf{m}}$.  (Crucially, for reading the following theorem, the indexing of the reduction vector was reversed between the preprint~\cite{CHT09} and its publication~\cite{CHT11}).
\begin{corollary}\cite[Theorem 4.2.2]{CHT09}\label{cor:predv}
	If $\tI_{\Omega}^{\mathbf{m}}=\cap_{\sigma\in\Omega}\wp_{\sigma}^{m_\sigma}$ is an ideal of fat points in $\Proj^2$ which has a positive full reduction vector $\mathbf{d} =(d_1,\dots, d_n)$, then 
	\[ n +\min \bigl\{d_1-n,d_2-n+1,\ldots, d_n-1, 0\bigr\}\leq \alpha \bigl(\tI_{\Omega}^{\mathbf{m}}\bigr)\leq n\ .
	\]
\end{corollary}

\begin{proposition}\label{prop:GenericWaldschmidt}
	Suppose $\Delta$ is a closed vertex star with interior vertex $\gamma$ and the two-faces of $\Delta$ all span distinct planes.  Let $\X=\{P_\sigma\}_{\sigma\in\Delta^\circ_2}$ be the set of points dual to the collection of forms $\{\ell_{\sigma}\}_{\sigma\in\Delta^\circ_2}$.  Then $\widehat{\alpha}(\tI_\X)\ge\max\bigl\{f^\circ_1/2,3\bigr\}$.
\end{proposition}
\begin{proof}
	Since the two-faces of $\Delta$ all span distinct planes, the set $\X=\{P_\sigma\}_{\sigma\in\Delta^\circ_2}$ of points dual to the forms $\{\ell_{\sigma}\}$ are all distinct.  Choose any ordering of the interior one-faces of $\Delta$: so $\Delta^\circ_1=\{\tau_1,\ldots,\tau_{f^\circ_1}\}$.  Let $n_i=n_{\tau_i}$ be the number of faces $\sigma\in\Delta^\circ_2$ which contain $\tau_i$.  Dually this gives lines $L_1,\ldots,L_{f^\circ_1}$ in $\Proj^2$ so that line $L_i$ contains $n_i$ many points of $\X$.  Moreover, exactly two interior one-faces are contained in every interior two-face of $\Delta$.  Hence each point $P_\sigma\in\X$ is at the intersection of exactly two lines from the set $\{L_1,\ldots,L_{f^\circ_1}\}$.  See Example~\ref{Ex:genOctDual} for an illustration of these properties.
	
	We define a full reduction vector $\mathbf{d}=\bigl(d_1,\ldots,d_{sf^\circ_1}\bigr)$ of length $sf^\circ_1$ for $\tI_\X^{(2s)}$ (where each point $P_\sigma$ has multiplicity $2s$) as follows.  This reduction vector is obtained with the sequence of lines $\bigl(L_1,L_2,\ldots,L_{f^\circ_1}\bigr)$ repeated $s$ times (in order).  Since every point $P_\sigma$ is on exactly two of the lines from $\bigl\{L_1,\cdots,L_{f^\circ_1}\bigr\}$, every time the sequence $(L_1,\ldots,L_{f^\circ_1})$ is completed the multiplicity of every point is reduced by two (this is why the entire reduction vector has length $sf^\circ_1$).  On the $(k+1)$st repetition of the sequence $\{L_1,\ldots,L_{f^\circ_1}\}$, the entry $d_{kf^\circ_1+i}$ (corresponding to the $(k+1)$st time the line $L_i$ is repeated), satisfies 
	\begin{equation}\label{eq:RedVector}
	d_{kf^\circ_1+i}\ge (2(s-k)-1)n_i \mbox{ for } 0\le k\le s-1, 1\le i\le f^\circ_1.
	\end{equation}
	This also shows that the reduction vector is positive, so we may apply Corollary~\ref{cor:predv}.  See Example~\ref{Ex:genOctDual}, where the case $s=2$ is worked out for the generic centrally triangulated octahedron.
	
	By Corollary~\ref{cor:predv}, we have that $\alpha\bigl(\tI^{(2s)}_\X\bigr)\ge n+\min\bigl\{d_1-n,d_2-n+1,\ldots,d_n-1,0\bigr\}$, where $n=sf^\circ_1$.  As in the previous paragraph, we consider the reduction vector indexed in the form $d_{kf^\circ_1+i}$, where $0\le k\le s-1$ and $1\le i\le f^\circ_1$. Now
	\[
	\begin{array}{rll}
	sf^\circ_1+d_{kf^\circ_1+i}-(sf^\circ_1-kf^\circ_1-i+1)=& d_{kf^\circ_1+i}+kf^\circ_1+i-1 & \\
	\ge &  (2(s-k)-1)n_i+kf^\circ_1+i-1 & \mbox{by~\eqref{eq:RedVector}}\\
	\ge & 6(s-k)-3+kf^\circ_1 & \mbox{since } n_i\ge 3, i\ge 1\\
	= & 6s-3+k(f^\circ_1-6). &
	\end{array}
	\]
	\noindent So $\alpha\bigl(\tI^{(2s)}\bigr)\ge \max\bigr\{6s-3,(s-1)f^\circ_1+3\bigr\}$ (the maximum depends on whether $f^\circ_1\ge 6$ or $f^\circ_1< 6$) and
	\[
	\widehat{\alpha}(\tI_\X)=\lim_{s\to\infty}\frac{\alpha(\tI^{(2s)})}{2s}\ge \lim_{s\to\infty}\frac{\max\{6s-3,(s-1)f^\circ_1+3\}}{2s}=\max\bigl\{3,f^\circ_1/2\bigr\},
	\]
	proving the proposition.
\end{proof}

\begin{remark}
	Notice that Chudnovsky's bound $\walpha(\tI_\X)\ge\frac{\alpha(\tI_\X)+1}{2}$ does not suffice to prove Proposition~\ref{prop:GenericWaldschmidt}.  Consider the case $f^\circ_1=5$.  In this case $f^\circ_2=9$ (see Lemma~\ref{lem:5Vert}), so the dual set of points consists of $9$ points.  These $9$ points necessarily lie on a cubic (but it is easily checked that they do not lie on any conic), so $\alpha(\tI_\X)=3$.  Chudnovsky's bound thus gives $\walpha(\tI_{\X})\ge 2$ but not $\walpha(\tI_\X)\ge\frac{5}{2}$.  Nevertheless many (probably most) cases of Proposition~\ref{prop:GenericWaldschmidt} do follow from Chudnovsky's bound -- whenever the dual set of points $\X$ does not lie on a curve of degree four the result follows immediately.
\end{remark}

\begin{corollary}\label{cor:regularity}
	Suppose $\Delta$ is a closed vertex star with interior vertex $\gamma$ so that the span of every two-dimensional face is distinct.  Then 
	\[
	\dim \tJ(\gamma)_d=\binom{d+2}{2}\mbox{ for }
	\left\lbrace
	\begin{array}{ll}
	d>2r & \mbox{ if } f^\circ_1=4\\[7 pt]
	d>\dfrac{5r}{3} & \mbox{ if } f^\circ_1=5\\[10 pt]
	d>\dfrac{3r}{2} & \mbox{ if } f^\circ_1\ge 6
	\end{array}
	\right.
	.
	\]
\end{corollary}
\begin{proof}
	Let $\X=\{P_\sigma\}_{\sigma\in\Delta^\circ_2}$.  By assumption, all the points of $\X$ are distinct and Proposition~\ref{prop:GenericWaldschmidt} applies.  It follows readily from Proposition~\ref{prop:regularitybound} that if $\walpha(\tI_X)\ge M$, where $M>1$, then $\dim \tJ(\gamma)_d=\binom{d+2}{2}$ for $d>\frac{Mr}{M-1}$.  Since $f^\circ_1\ge 4$, the result is immediate from Proposition~\ref{prop:GenericWaldschmidt}.
\end{proof}

\section{Proof of Theorem~\ref{thm:LBGenericClosedVertexStars}: lower bound for splines on vertex stars}\label{sec:bounds_cells}
In this section we will prove Theorem~\ref{thm:LBGenericClosedVertexStars}.  We use Equation~\eqref{eq:celldimformula} from Proposition~\ref{prop:EulerCharacteristicAndDimension}, so we first explain how to compute $\chi(\calJ[\Delta],d)$ when $\Delta$ is the star of a vertex.  If $\Delta$ is a closed vertex star with interior vertex $\gamma$ then the Euler characteristic of $\calJ[\Delta]$ has the form
\begin{equation}\label{eq:EulerCClosedStar}
\chi(\calJ[\Delta],d)=\sum_{\sigma\in\Delta^\circ_2} \dim \tJ(\sigma)_d-\sum_{\tau\in\Delta^\circ_1} \dim\tJ(\gamma)_d+\dim \tJ(\gamma)_d
\end{equation}
If $\Delta$ is an open vertex star, then the Euler characteristic of $\calJ[\Delta]$ has the form
\begin{equation}\label{eq:EulerCOpenStar}
\chi(\calJ[\Delta],d)=\sum_{\sigma\in\Delta^\circ_2} \dim \tJ(\sigma)_d-\sum_{\tau\in\Delta^\circ_1} \dim\tJ(\gamma)_d
\end{equation}

We use the following notation in the formulas for $\dim \tJ(\tau)_d$.

\begin{notation}[Data attached to edges]\label{not:EdgeData}
	For a given $r\geq 0$ and $\tau\in\Delta_1$, 
	\begin{itemize}
		\renewcommand{\labelitemi}{\scriptsize$\blacksquare$}
		\item we define $t_\tau=\min\{n_\tau,r+2\}$, where  $n_\tau=\#\{\sigma\in\Delta_2:\tau\subset\sigma\}$ is the number of 2-dimensional faces having $\tau$ as an edge;
		\item and the constants 
		\[q_\tau = \biggl\lfloor \frac{t_\tau(r+1)}{t_\tau-1}\biggr\rfloor,\quad  a_\tau=t_\tau(r+1)-(t_\tau-1) q_\tau\ , \text{ and}\quad
		b_\tau=t_\tau-1-a_\tau.\]
		Notice that $t_\tau(r+1)=q_\tau(t_\tau-1)+a_\tau$ i.e.,   $q_\tau,a_\tau$ 
		are the quotient and remainder obtained when dividing  $t_\tau(r+1)$ by $t_\tau-1$.
	\end{itemize}
\end{notation}


For the following proposition, recall that we use the convention $\binom{a}{b}=0$ when $a<b$.



\begin{proposition}\label{prop:DimensionsCountsForStars}
	Suppose $\Delta\subset\R^3$ is the star of a vertex $\gamma$, $r\ge 0$ is an integer, $\sigma\in\Delta_2$, and $\tau\in\Delta_1$.  Let $n_\tau,t_\tau,q_\tau,a_\tau,$ and $b_\tau,$ be as in Notation~\ref{not:EdgeData}, and $D_\gamma$ as in~\eqref{eq:Dgamma}.  Then
	
	\begin{align}
	\dim \tS_d= & \binom{d+2}{2}\nonumber\\[10 pt]
	\dim \tJ(\sigma)_d= & \binom{d+1-r}{2}\nonumber\\[10 pt]
	\label{eq:3varcodim2}
	\dim \tJ(\tau)_d\ge & t_\tau\binom{d+1-r}{2}-a_\tau\binom{d+1-q_\tau}{2}-b_\tau\binom{d+2-q_\tau}{2}\\[10 pt]
	\label{eq:3varcodim3}
	\dim \tJ(\gamma)_d\le & \binom{d+2}{2}, \mbox{ with equality for } d>D_\gamma \\[10 pt]
	\label{eq:FiniteLengthH1}
	\dim H_1(\calJ[\Delta])_d= & 0\quad \mbox{ for } d\gg 0
	\end{align}
	If $n_\tau$ is replaced by the maximum number of $2$-faces $\sigma$ containing $\tau$ so that $\ell_\sigma$ is distinct (i.e. if we set $n_\tau$ to be the number of distinct planes surrounding the edge $\tau$), then the inequality~\eqref{eq:3varcodim2} is an equality.  In particular, if $\Delta$ is generic with respect to $r$ and $d$ then~\eqref{eq:3varcodim2} is an equality.
\end{proposition}
\begin{remark}
	The computation of $\dim J(\tau)_d$ follows from an argument of Schumaker~\cite{SchumakerLower}, as indicated by Schenck in~\cite{MinReg}.  The formulation we give is equivalent to expressions derived by Schenck~\cite{MinReg} and Geramita and Schenck~\cite{FatPoints}, although it is expressed slightly differently.
\end{remark}
\begin{proof}
	The computations for $\dim \tS_d$ and $\dim \tJ(\sigma)_d$ are straightforward.  The computation of $\dim \tJ(\tau)_d$ follows from~\cite{SchumakerLower} as indicated by Schenck in~\cite{MinReg}.  It also follows readily from apolarity (particularly \ref{theor:iarrobino}) as shown by Geramita and Schenck in~\cite{FatPoints}.  The inequality~\eqref{eq:3varcodim3} follows since $\tJ(\gamma)\subset \tS_d$, so $\dim \tJ(\gamma)_d\le \dim\tS_d=\binom{d+2}{2}$.  The equality in~\eqref{eq:3varcodim3} for $d>D_\gamma$ follows from Corollary~\ref{cor:regularity}.  Equation~\eqref{eq:FiniteLengthH1} follows from~\cite[Lemma~3.2]{LCoho} or~\cite[Lemma~3.1]{Spect}.
\end{proof}

\noindent If $\Delta$ is a closed vertex star we define
\begin{multline}\label{eq:LBclosedstar}
\LBcs(\Delta,d,r):= 2\binom{d+2}{2}+\left(f^\circ_2-\sum\limits_{\tau\in\Delta^\circ_1} t_\tau \right)\binom{d+1-r}{2}\\ +\sum\limits_{\tau\in\Delta^\circ_1} \left( a_\tau\binom{d+1-q_\tau}{2}+b_\tau\binom{d+2-q_\tau}{2}\right).
\end{multline}
We write $\LBcs(d)$ instead of $\LBcs(\Delta,d,r)$ if $\Delta,r$ are understood.  If $\Delta$ is an \textit{open} vertex star we define
\begin{multline}\label{eq:LBopenstar}
\LBos(\Delta,d,r) := \binom{d+2}{2}+\left(f^\circ_2-\sum\limits_{\tau\in\Delta^\circ_1} t_\tau \right)\binom{d+1-r}{2}\\[10 pt]
+\sum\limits_{\tau\in\Delta^\circ_1} \left( a_\tau\binom{d+1-q_\tau}{2}+b_\tau\binom{d+2-q_\tau}{2}\right).
\end{multline}
Again we write $\LBos(d)$ if $\Delta,r$ are understood.

\begin{proposition}\label{prop:EulerCharBounds}
	Suppose $\Delta$ is a generic vertex star.  If $\Delta$ is an open vertex star, then
	\[
	\LBos(\Delta,d,r)=\chi(\calR/\calJ,d)=\binom{d+2}{2}+\chi(\calJ,d)
	\]
	for every integer $d\ge 0$.  If $\Delta$ is a closed vertex star with interior vertex $\gamma$ then
	\[
	\LBcs(\Delta,d,r)=\binom{d+2}{2}-\dim\tJ(\gamma)_d+\chi(\calR/\calJ,d)\, =\,  2\dbinom{d+2}{2}-\dim \tJ(\gamma)_d+\chi(\calJ,d)
	\]
	If $d>D_\gamma$ then 
	\[
	\LBcs(\Delta,d,r)=\chi(\calR/\calJ,d)=\dbinom{d+2}{2}+\chi(\calJ,d).
	\]
\end{proposition}
\begin{proof}
	These follow readily from~\eqref{eq:EulerCClosedStar},~\eqref{eq:EulerCOpenStar}, and Proposition~\ref{prop:DimensionsCountsForStars}.
\end{proof}

\begin{remark}\label{rem:LowerBoundForOpenStar}
	If $\Delta$ is an open star it is well-known that $\LBos(\Delta,d,r)\le \dim \hspl^r_d(\Delta)$ (this follows from Proposition~\ref{prop:EulerCharBounds} and~\eqref{eq:FiniteLengthH1}) with equality if $d\ge 3r+2$ and vertices are positioned generically.  See~\cite[Theorem~3]{ANS96} and~\cite{SchumakerLower}; this essentially reduces to the planar case.
\end{remark}

\begin{remark}\label{rem:EquivalenceOfFormulas}
	We discuss how to show that the the formulas $\LBcs(\Delta,d,r)$ and $\LBos(\Delta,d,r)$ are the formulas appearing in Equations~15 and~16 of~\cite{ANS96}, as claimed in the introduction.  We do this for $\LBcs(\Delta,d,r)$; the computation for $\LBos(\Delta,d,r)$ is similar.  If $r\in \R$ is any real number, put $[r]_+ := \max\{0,r\}$.  Let $\sigma=\sum_{\tau\in\Delta_1}\sum_{j=1}^{d-r}[r+j+1-n_\tau j]_+$.  Then we can re-write $\LBcs(\Delta,d,r)$ as $\LBcs(\Delta,d,r)=2\binom{d+2}{2}+f^\circ_2\binom{d+1-r}{2}-f^\circ_1\left(\binom{d+2}{2}-\binom{r+2}{2}\right)+\sigma$.  Using the relation $f^\circ_2=3f^\circ_1-6$ allows us to write $\LBcs(\Delta,d,r)$ completely in terms of $f^\circ_1$:
	\[
	\LBcs(\Delta,d,r)=(d-r)(d-2r)f^\circ_1-2d^2+6dr-3r^2+3r+2+\sigma.
	\]
	If each edge is surrounded by two-faces which span distinct planes, then this is exactly the expression that appears in Equation~15 of~\cite{ANS96}.  Otherwise Equation~15 of~\cite{ANS96} will be slightly larger.  In general, Equations~15 and~16 of~\cite{ANS96} coincide with the graded Euler characteristic $\chi(\calR/\calJ,d)$.
\end{remark}


\begin{proof}[Proof of Theorem~\ref{thm:LBGenericClosedVertexStars}]
	We assume that $\Delta$ is a generic closed vertex star.  By Proposition~\ref{prop:EulerCharacteristicAndDimension},
	\[
	\dim \hspl^r_d(\Delta)=\dim \tS_d+\chi(\calJ[\Delta],d)+\dim H_1(\calJ[\Delta])_d.
	\]
	If $d>D_\gamma$ then
	\[
	\dim \hspl^r_d(\Delta)=\LBcs(\Delta,d,r)+\dim H_1(\calJ[\Delta])_d
	\]
	by Proposition~\ref{prop:EulerCharBounds}.  Thus if $d>D_\gamma$, $\dim \hspl^r_d(\Delta)\ge \max\{\binom{d+2}{2},\LBcs(\Delta,d,r)\}$.  If $\Delta$ is not generic, the conclusion follows by Lemma~\ref{lem:Generic}.
	
	The second statement of Theorem~\ref{thm:LBGenericClosedVertexStars} follows since $\dim \spl^r_d(\Delta)=\sum_{i=0}^d \dim\hspl^r_d(\Delta)$ and we may always take $\binom{i+2}{2}$ as a lower bound for $\dim \hspl^r_i(\Delta)$.
\end{proof}

\section{Proof of Theorem~\ref{thm:WhitelyGenericLowDegree}: low degree splines on generic closed vertex stars}\label{sec:GenericLowDegree}
The main case of Theorem~\ref{thm:WhitelyGenericLowDegree}, namely when $\Delta$ is a closed vertex star with $f^\circ_1\ge 6$, is a slight modification of a result of Whiteley~\cite{WhiteleyComb}.  Its proof relies on techniques from rigidity theory, which are explained in detail in several articles of Whiteley~\cite{WhiteleyComb,WhiteleyM,Whiteley90} as well as the article by Alfeld, Schumaker, and Whiteley~\cite[Theorems~27, 33]{ASWTet}.  Thus if $f^\circ_1\ge 6$ we will only go into enough detail to explain the alterations needed in Whiteley's argument.  The cases $f^\circ_1=4$ and $f^\circ_1=5$ require a bit more delicacy.

We introduce a helpful auxiliary construction.  Suppose $\Delta$ is a tetrahedral vertex star.  There is a natural graph associated to $\Delta$ which we call \textit{the graph of} $\Delta$ and write as $G_\Delta$.  The graph $G_\Delta$ is constructed from $\Delta$ as follows: the vertices are the interior edges, i.e. $V=\Delta^\circ_1$, and the edges correspond to the interior faces of dimension two, i.e. $E=\Delta^\circ_2$.  The combinatorics of $\Delta$ are often easiest to detect from this graph.

\begin{remark}
	Let $\gamma$ be the vertex at which all tetrahedra of $\Delta$ meet and assume that $\gamma$ is at the origin in $\R^3$.  If we assume all other vertices of $\Delta$ lie on a unit sphere centered at the origin, then $G_\Delta$ is simply the \textit{edge graph} of the simplicial polytope formed by taking the convex hull of the vertices of $\Delta$.  In fact, scaling the vertices of $\Delta$ clearly does not affect $G_\Delta$, so we may as well assume that $\Delta$ is a barycentric subdivision of a simplicial polytope, and $G_\Delta$ is the edge graph of this simplicial polytope.
\end{remark}

The following characterization of the graphs which may arise as the graph of a closed vertex star is due to Steinitz (see~\cite[Chapter~3]{Zie}).  A graph is called $d$-connected if it remains connected after removing any set of $(d-1)$ vertices and their incident edges.

\begin{theorem}[Steinitz]\label{thm:Steinitz}
	$G_\Delta$ is the graph of a closed vertex star if and only if it is simple, planar, and $3$-connected.
\end{theorem}

\begin{lemma}\label{lem:5Vert}
	Suppose $\Delta$ is a closed tetrahedral star.  If $f^\circ_1=4$ then $G_\Delta$ is the complete graph on $4$ vertices, so $\Delta$ is the Alfeld split of a tetrahedron.  If $f^\circ_1=5$ then $G_\Delta$ must be the graph shown on the left in Figure~\ref{fig:5Vert}, and $\Delta$ is the barycentric subdivision of a triangular bipyramid.
\end{lemma}

\begin{proof}
	If $f^\circ_1=4$ the result is clear, so we assume $f^\circ_1=5$.  Euler's formula combined with the fact that each vertex must have degree at least $3$ gives two possible degree sequences of simple planar $3$-connected graphs: $(3,3,3,3,4)$ or $(3,3,4,4,4)$.  There is precisely one graph realizing each of these degree sequences - those pictured in Figure~\ref{fig:5Vert}.  Clearly only the one on the left is simplicial.
\end{proof}
\begin{figure}[htbp]
	\centering
	\includegraphics[height=2.9cm]{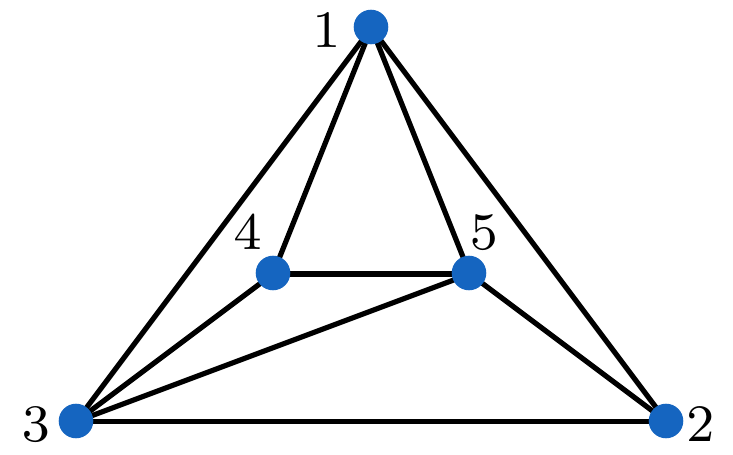} 
	\qquad 
	\includegraphics[height=3.1cm]{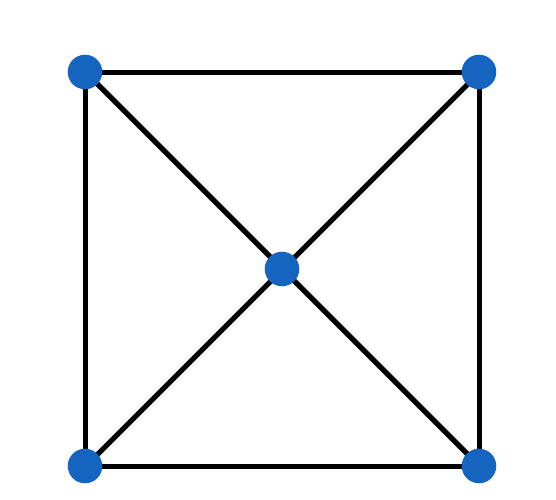}
	\caption{Simple $3$-connected planar graphs with five vertices}\label{fig:5Vert}
\end{figure}
\begin{proof}[Proof of Theorem~\ref{thm:WhitelyGenericLowDegree}]
	
	\noindent\textbf{The case $f^\circ_1=4$ and $d\le 2r$}: By Lemma~\ref{lem:5Vert}, if $f^\circ_1=4$ then $\Delta$ must be the Alfeld split of a tetrahedron.  It then follows from a recent result of Schenck~\cite{S14} that $\dim \hspl^r_d(\Delta)=\binom{d+2}{2}$ for $d\le 2r$.
	
	\noindent\textbf{The case $f^\circ_1=5$ and $d\le (5r+2)/3$}: We prove this in Proposition~\ref{prop:5Vert}.
	
	\noindent\textbf{The case $d\le (3r+1)/2$:} 
	As indicated above, we only summarize the broad strokes.  In~\cite[Corollary~7]{WhiteleyComb}, Whiteley proves that if $\Delta$ is a generic rectilinear partition \textit{with a triangular boundary} then $\Delta$ has only trivial $r$-splines for $d\le (3r+1)/2$.  The proof given in~\cite{WhiteleyComb} needs only minor modifications to yield the case $d\le (3r+1)/2$ in Theorem~\ref{thm:WhitelyGenericLowDegree}.
	
	
	Whiteley proves~\cite[Corollary~7]{WhiteleyComb} by induction, starting with the base case of the Alfeld (or Clough-Tocher) split of a triangle.  Whiteley proves in~\cite[Lemma~5]{WhiteleyComb} that every triangulated triangle can be produced from the Alfeld (or \textit{Clough-Tocher}) split of a triangle by a sequence of \textit{vertex splits}.  Hence the graph $G_\Delta$ of any tetrahedral vertex star $\Delta$ can be produced by a sequence of vertex splits.  See also~\cite[Lemma~29]{ASWTet}.
	
	The corresponding induction for tetrahedral vertex stars has as its base case the Alfeld (or Clough-Tocher) split of a tetrahedron ($G_\Delta$ is the complete graph on four vertices).  Moreover, the process of \textit{vertex splitting} used by Whiteley naturally extends to \textit{edge splitting} for star complexes (simply `cone over' the vertex split).  Then~\cite[Lemma~5]{WhiteleyComb}, translated to closed tetrahedral vertex stars, says that every closed tetrahedral star can be produced from the Alfeld split of a tetrahedron by a sequence of edge splits.
	
	Whiteley's induction in the planar case then proceeds as follows.  The base case is the Alfeld split of a triangle, which has no non-trivial $C^r$ splines in degree at most $(3r+1)/2$.  Now suppose $\Delta$ is a triangulated triangle and $\Delta'$ is obtained from $\Delta$ by a vertex split.  Whiteley shows that the procedure of vertex splitting (for most choices of coordinates for the new vertex) cannot produce a non-trivial $C^r$ spline in degree at most $(3r+1)/2$ for $\Delta'$ if $\Delta$ has no non-trivial splines in degree at most $(3r+1)/2$.  This completes the induction.
	
	We translate this to our setting as follows.  The base case is the Alfeld split of a tetrahedron; we saw above that there are no non-trivial splines on the Alfeld split of a tetrahedron of degree at most $2r$.  Now suppose that $\Delta$ is a tetrahedral vertex star and $\Delta'$ is obtained from $\Delta$ by an edge split - the analog of Whiteley's vertex splitting.  As Whiteley shows in~\cite[Corollary~7]{WhiteleyComb}, edge-splitting cannot produce a non-trivial $C^r$ spline in degree at most $(3r+1)/2$ for $\Delta'$ if $\Delta$ has no non-trivial splines in degree at most $(3r+1)/2$ (for most choices of edges splitting the original edge).  Thus induction completes the argument.
\end{proof}

\begin{remark}
	Alternatively, the case $d\le (3r+1)/2$ of Theorem~\ref{thm:WhitelyGenericLowDegree} could be proved (following the pattern of~\cite{ASWTet}) by considering splines on closed \textit{generalized triangulations} which occur as projections of closed tetrahedral vertex stars.  Whiteley's arguments from~\cite{WhiteleyComb} could then be extended to generalized triangulations, and the result lifted to three dimensions by~\cite[Theorem~37]{ASWTet}.
\end{remark}

\section{Examples}\label{sec:Examples}

In this section we illustrate our bounds in several examples.  Accompanying code for these and other examples can be found under the research tab at the first author's website:~\url{https://midipasq.github.io/}.

\subsection{Generic bipyramid}\label{ss:genbi}
Let $\Delta$ be the vertex star with interior vertex $\gamma$ in Figure \ref{fig:bipyramid}, where the vertex coordinates are chosen \textit{generically}. 
The number of interior two-dimensional faces is $f_2^\circ= 15$, and the number of interior edges is $f_1^\circ = 7$. Each of the five edges in the base of the bipyramid have four 2-dimensional interior faces attached to them, i.e., if we denote them by $\tau\in\Delta_0^\circ$ then $n_{\tau}= 4$. 
The latter implies $t_{\tau}= r+2$ if $r=0, 1, 2$, and $t_{\tau}= 4$ for every $r\geq 3$.
Similarly, for the edges $\tau'$ and $\tau'$ connecting $\gamma$ to the top and the bottom of $\Delta$, we have
$n_{\tau'}=n_{\tau''}=5$. 
Thus, $t_{\tau'}=t_{\tau''}=r+2$ for $r=0,1, 2$, and $t_{\tau}=t_{\tau''}= 5$ for every $r\geq 3$.  

\begin{figure}[htbp]
	\centering	
	\includegraphics[scale=2.2]{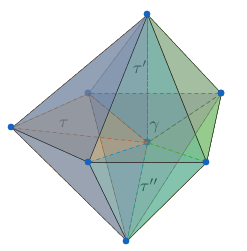}\quad \includegraphics[scale=2.2]{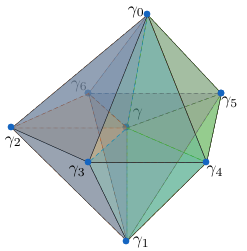}
	\caption{Generic (left) and non-generic (right) bipyramids over a pentagon.}\label{fig:bipyramid}
\end{figure}

By Theorem \ref{thm:LBGenericClosedVertexStars}, $\dim \hspl^r_d(\Delta)\ge \max\left\{\binom{d+2}{2}, \LBcs(\Delta)\right\}$ for $d> D_\gamma$, where $D_\gamma=\lfloor (3r+1)/2\rfloor$ and 
\begin{multline*}
\LBcs(\Delta,d,r)= 2\binom{d+2}{2}+\bigl(15-5t_{\tau}-2t_{\tau'}\bigr)\binom{d+1-r}{2}
+  5a_{\tau} \binom{d+1-q_{\tau}}{2}\\
+  5b_{\tau}\binom{d+2-q_{\tau}}{2} +  2a_{\tau'} \binom{d+1-q_{\tau'}}{2}+ 2b_{\tau'}\binom{d+2-q_{\tau'}}{2} \ . 
\end{multline*}
For instance, if $r\leq 2$, then $q_\tau= q_{\tau'}=q_{\tau''} = r+2$, and therefore $a_\tau=a_{\tau'}=a_{\tau''}=0$ and $b_\tau=b_{\tau'}=b_{\tau''}= r+1$. Thus,
\begin{align*}
\LBcs(\Delta,d,r)&= 2\binom{d+2}{2}+\bigl(15-7(r+2)\bigr)\binom{d+1-r}{2}
+  7(r+1)\binom{d-r}{2} \\[0.5em]
&= 5d^2 - 15dr + 11r^2 + 3r + 2\ , \text{ if $d\geq r$},
\intertext{and for $d\geq D_\gamma$,}	
\dim \spl^r_d(\Delta)&\ge \binom{D_\gamma +3}{3} +\sum_{k=D_\gamma +1}^d \max\left\lbrace\binom{k+2}{2},\LBcs(\Delta,k,r)\right\rbrace\ .
\end{align*}
In the case $r=2$, we have $D_\gamma=3$, and for $d\ge 3$ the lower bound on the dimension of the spline space is given by
\[
\dim \spl^2_d(\Delta)\ge 20 +\sum_{k=4}^d \LBcs(\Delta,k,2) = \frac{5}{3}d^3-\frac{25}{2}d^2+\frac{227}{6}d-26\ .
\]
We list some numerical values of $\LBcs(d)$ in Table~\ref{tbl:Bipyramid} for $1\le r\le 4$ and various $d$.  Table~\ref{tbl:Bipyramid} also compares the values of $\LBcs(d)$ to the actual value of $\dim \hspl^r_d(\Delta)$ for generic vertex positions, listed under the \textit{gendim} column.  The value $d=D_\gamma$ appears in bold for each $r$.  These computations were performed using the \textit{AlgebraicSplines} package in Macaulay2~\cite{M2}.

\subsection{Non-generic bipyramid.}\label{ss:nongen}
Our arguments can be modified to produce better bounds in non-generic situations.  We illustrate with special vertex positions for the example of the bipyramid over a pentagon -- see the configuration on the right in Figure~\ref{fig:bipyramid}.  Label the vertices as indicated on the right in Figure~\ref{fig:bipyramid}.  We assume that $\gamma_2,\gamma_3,\gamma_4,\gamma_5,$ and $\gamma_6$ all lie in the $xy$-plane (such configurations are studied in~\cite{CDS16}).  Assume further that $\gamma,\gamma_0,\gamma_1,$ and $\gamma_i$ are not coplanar for $i=2,\ldots,6$ and $\gamma,\gamma_i,$ and $\gamma_j$ are not colinear for any $2\le i<j\le 6$.  We write $\Delta$ for this non-generic vertex star.

The collection $\X=\{\wp_\sigma:\sigma\in\Delta^\circ_0\}$ of points dual to $\{\sigma:\sigma\in\Delta^\circ_0\}$ consists of $11$ points.  The five two-faces with vertices $\gamma,\gamma_i,\gamma_{i+1}$ for $i=2,3,4,5,6$ (indices taken cyclically in this set) all span the same plane, so all correspond to the same dual point.  Our assumptions for the rest of the vertices ensure that the remaining ten two-faces all span distinct planes, hence give rise to distinct dual points.  Write $L_0,L_1$ for the linear forms defining the lines dual to the edges with vertices $\gamma,\gamma_0$ and $\gamma,\gamma_1$ in $\Delta$.  Write $L_i$ for the linear form defining the line dual to the edge with vertices $\gamma,\gamma_i$ for $i=2,\ldots,6$.  The set $\X$ decomposes as a union of $5$ points which lie on $L_0$, $5$ points which lie on $L_1$, and the isolated point $v=[0:0:1]$.  Since the dual points do not lie on any conic, it follows from Chudnovsky's bound that $\walpha(\tI_{\X})\ge\frac{\alpha(\tI_\X+1)}{2}=2$.  It follows from Proposition~\ref{prop:regularitybound} that $\dim \tJ(\gamma)_d=\binom{d+2}{2}$ for $d>2r$.

\begin{remark}
	A more careful analysis shows that in fact $\walpha(\tI_{\X})=\frac{13}{5}$ and $\dim \tJ(\gamma)_d=\binom{d+2}{2}$ for $d>\frac{13}{8}r$.  However we will see this more careful analysis is unnecessary.
\end{remark}

Denote by $\LBcs_{1}(\Delta,d,r)$ the expression which results if we replace $n_\tau$ by the number of distinct planes surrounding the edge $\tau$ (and thus replace $t_\tau$ by the minimum of $r+2$ and the number of distinct planes surrounding $\tau$) in $\LBcs(\Delta,d,r)$.  It is shown in~\cite{ANS96} that $\dim \hspl^r_d(\Delta)=\LBcs_1(\Delta,d,r)$ for $d\ge 3r+2$.  

From the calculation above, $\LBcs_1(\Delta,d,r)=\chi(\calR/\calJ,d)$ for $d>2r$, so in particular $\LBcs_1(\Delta,d,r)\le \dim \hspl^r_d(\Delta)$ for $d>2r$.  Now put $f(d,r)=\binom{d+2}{2}+\binom{d+1-r}{2}$.  Since the plane $z=0$ cuts straight through $\Delta$, the spline $F$ which evaluates to $z^{r+1}$ on every upper tetrahedron and $0$ on every lower tetrahedron is in $\hspl^r_{r+1}(\Delta)$.  It follows that $f(d,r)\le \hspl^r_d(\Delta)$ for any $d\ge 0$.  Mimicking Theorem~\ref{thm:LBGenericClosedVertexStars}, we can take $f(d,r)$ as a lower bound on $\dim \hspl^r_d(\Delta)$ when $d\le 2r$ and $\max\{f(d,r),\LBcs_1(\Delta,d,r)\}$ as the lower bound on $\dim \hspl^r_d(\Delta)$ when $d>2r$.  In Table~\ref{tbl:Bipyramid} values of $f(d,r)$ and $\LBcs_1(\Delta,d,r)$ are listed for $1\le r\le 4$ and various $d$.  These are compared to the actual dimension of $\hspl^r_d(\Delta)$, the values of which are in the column labeled \textit{symdim}.  (Again, these values were computed using the AlgebraicSplines package in Macaulay2.)

How optimal is this lower bound?  In particular, is it possible to improve the lower bound by using $\chi(\calR/\calJ,d)$ (and thus a computation of $\dim \tJ(\gamma)_d$) when $d\le 2r$?  We show the answer is no.  First of all, an application of the upper bound from~\cite[Theorem~4.1]{D3} shows that in fact $\dim \hspl^r_d(\Delta)=f(r,d)$ for $d\le \lfloor\frac{3r+1}{2}\rfloor=D_\gamma$.  This gives a range of degrees $D_\gamma< d\le 2r$ where it might be possible to improve the lower bound by using $\chi(\calR/\calJ,d)$ instead of $f(r,d)$.  Since $\chi(\calR/\calJ,d)\le \LBcs_1(\Delta,d,r)$, if we show that $\LBcs_1(\Delta,d,r)\le f(r,d)$ for this range of values, then we have shown that we cannot improve the lower bound by using $\chi(\calR/\calJ,d)$ for values of $d$ which are smaller than $2r$.

In what follows we assume $r=4k-1$ and $k\ge 1$ to simplify calculations.  We compute that
\[
\LBcs_1(d)=\LBcs_1(\Delta,d,4k-2)=2\binom{d+2}{2}-10\binom{d+2-4k}{2}+10\binom{d+2-6k}{2}+2\binom{d+2-5k}{2}.
\]
For $d\ge 6k-2$,
\[
\LBcs_1(d)=5d^2 - 15(4k - 1)d + 200k^2 - 90k + 10.
\]
Also for $d\ge 4k-2$,
\[
f(d)=f(d,4k-1)=d^2 - (4k - 3)d + 8k^2 - 6k + 2.
\]
We can check that the polynomial $\LBcs_1(d)-f(d)$ attains a minimum of $-4k^2-1$ at $d=7k-3/2$.  Furthermore the roots of $\LBcs(d)-f(d)$ are
\[
d=7k-3/2\pm \sqrt{16k^2+4}.
\]
Thus $\LBcs(d)<f(d)$ for $6k-2\le d\le 11k-3/2$.  Notice this is long past the value of $d=8k-2$ where $\dim \tJ(\gamma)_d=\binom{d+2}{2}$ and thus $\LBcs(d)=\chi(\calR/\calJ,d)$!  So we cannot improve our lower bound by more careful computations of $\dim \tJ(\gamma)_d$ in degrees $d<2r$.  Similar arguments can be made for $r=4k,4k+1,4k+2$.

\begin{table}
	\renewcommand{\arraystretch}{1.2}
	\begin{tabular}{c|c||c|c|c||c|c|c}
		$r$&$d$&$\binom{d+2}{2}$&$\LBcs(\Delta,d,r)$&\mbox{gendim}&$\binom{d+2}{2}+\binom{d+1-r}{2}$&$\LBcs_1(\Delta,d,r)$&\mbox{symdim}\\
		
		\hline
		
		1&\textbf{2}&    6       &      6     		&     6       &    7     						&           6         &      7       \\
		1  & 3 &        10       &     16     		&    16       &   13     						&           16        &      16       \\
		1  & 4 &    	15  	 &     36     		&    36       &   21     						&           36        &      36      \\
		1  & 5 &    	    	 &     66     		&    66       &   31     						&           66        &      66       \\
		1  & 6 &    	    	 &    106     		&   106       &   43     						&          106        &     106       \\
		1  & 7 &    	    	 &    156     		&   156       &   57     						&          156        &     156       \\
		1  & 8 &    	    	 &    216     		&   216       &   73     						&          216        &     216       \\
		1  & 9 &    	    	 &    286     		&   286       &   91     						&          286        &     286       \\
		
		\hline
		
		2&\textbf{3}&	10		 &       7          &     10      &   11     						&        12           &      11       \\
		2  & 4 & 		15		 &      12          &     15      &   18     						&        17           &      18       \\
		2  & 5 & 		21		 &      27          &     27      &   27     						&        32           &      32       \\
		2  & 6 &  		28		 &      52          &     52      &   38     						&        57           &      57       \\
		2  & 7 &    			 &      87          &     87      &   51     						&        92           &      92       \\
		2  & 8 &   				 &     132          &    132      &   66     						&       137           &     137       \\
		2  & 9 &  				 &     187          &    187      &   83     						&       192           &     192       \\
		2  &10 & 				 &     252          &    252      &  102     						&       257           &     257       \\
		2  &11 &   				 &     327          &    327      &  123     						&       332           &     332       \\
		
		\hline
		
		3  & 4 &    	15		 &      15          &     15      &   16     						&        20           &      16       \\
		3&\textbf{5}&  	21		 &      15          &     21      &   24     						&        20           &      24       \\
		3  & 6 &  		28		 &      25          &     28      &   34     						&        30           &      34       \\
		3  & 7 &  		36		 &      45          &     45      &   46     						&        50           &      51       \\
		3  & 8 &  		45		 &      75          &     75      &   60     						&        80           &      80       \\
		3  & 9 &  				 &     115          &    115      &   76     						&       120           &      120      \\
		3  &10 &  				 &     165          &    165      &   94     						&       170           &      170      \\
		3  &11 &  				 &     225          &    225      &  114     						&       230           &      230      \\
		3  &12 &  				 &     295          &    295      &  136     						&       300           &      300      \\
		
		\hline
		
		4  & 5 &   		21		 &      27          &     21      &   22     						&         32          &      22       \\
		4&\textbf{6}&	28		 &      22          &     28      &   31     						&         32          &      31       \\
		4  & 7 &   		36		 &      27          &     36      &   42     						&         37          &      42       \\
		4  & 8 &   		45		 &      42          &     45      &   55     						&         52          &      56       \\
		4  & 9 &   		55		 &      67          &     67      &   70     						&         77          &      78       \\
		4  &10 &   		66		 &     102          &    102      &   87     						&        112          &      112      \\
		4  &11 &   				 &     147          &    147      &  106     						&        157          &      157      \\
		4  &12 &   				 &     202          &    202      &  127     						&        212          &      212      \\
		4  &13 &   				 &     267          &    267      &  150     						&        277          &      277      \\
		
	\end{tabular}
	
	\vspace{5 pt}
	
	\caption{Bounds for generic and non-generic bipyramids in Sections~\ref{ss:genbi} and~\ref{ss:nongen}}\label{tbl:Bipyramid}
\end{table}

\subsection{Non-simplicial vertex star.}\label{ss:cube}
For simplicity of exposition we have only considered the case where $\Delta$ is a simplicial vertex star.  However, Theorems~\ref{thm:LBGenericClosedVertexStars} and~\ref{thm:WhitelyGenericLowDegree} both hold verbatim if $\Delta$ is instead a \textit{polytopal} vertex star.  A polytopal vertex star is a collection of polytopes whose intersection contains a vertex $\gamma$ and satisfies that the intersection of each pair of polytopes is a face of both.  The main difference between splines on polytopal as opposed to simplicial vertex stars is that $\dim H_0(\calJ)_d$ may not vanish in large degree (see~\cite{TSchenck09}), however this is both highly \textit{non-generic} behavior and only makes $\dim \hspl^r_d(\Delta)$ larger.  Thus this behavior has no impact on whether $\LBcs(\Delta,d,r)$ is a lower bound on $\hspl^r_d(\Delta)$.

We briefly remark on the details that need to be checked to ensure that Theorems~\ref{thm:LBGenericClosedVertexStars} and~\ref{thm:WhitelyGenericLowDegree} carry over to polytopal vertex stars.  First, Theorem~\ref{thm:LBGenericClosedVertexStars} hinges on Proposition~\ref{prop:GenericWaldschmidt} and Corollary~\ref{cor:regularity}.  These easily carry over to polytopal vertex stars, as the simplicial nature of $\Delta$ plays no role in the proofs.  Now suppose $\Delta$ is a polytopal vertex star and $\Delta'$ is a triangulation of it which is also a simplicial vertex star (to justify the existence of such a triangulation takes a couple sentences, but it is not difficult).  Then $\hspl^r(\Delta)$ includes into $\hspl^r(\Delta')$.  By Theorem~\ref{thm:WhitelyGenericLowDegree} $\dim \hspl^r_d(\Delta')=\binom{d+2}{2}$ for $d\le D_\gamma$, hence $\dim \hspl^r_d(\Delta)=\binom{d+2}{2}$ for $d\le D_\gamma$ as well. 

We give a simple illustration.  Let $\Delta$ be the barycentric subdivision of a cube ($G_\Delta$ is shown in Figure~\ref{fig:Cube}).  Then $f^\circ_2=12,f^\circ_1=8,$ and $n_\tau=3$ for every interior edge (and so $t_\tau=3$ if $r\ge 1$).
\begin{figure}[htbp]
	\centering
	\includegraphics[height=3.0cm]{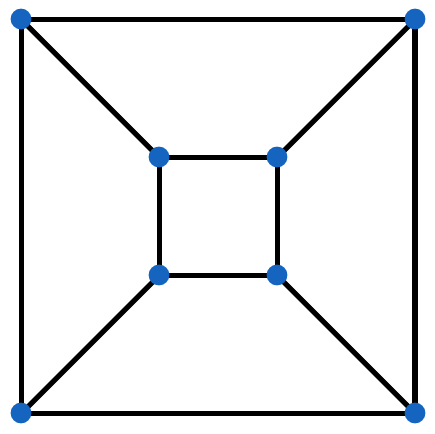}
	\caption{Graph of the barycentric subdivision of a cube.}\label{fig:Cube}
\end{figure}
We have
\[
\LBcs(\Delta,d,r)=2\binom{d+2}{2}-12\binom{d+1-r}{2}+8\left(a_\tau\binom{d+1-q_\tau}{2}+b_\tau \binom{d+2-q_\tau}{2}\right),
\]
where $q_\tau=\lfloor 3(r+1)/2\rfloor$, $a_\tau$ is the remainder when $3(r+1)$ is divided by two, and $b_\tau=2-a_\tau$.  If $r=2k-1$ then
\[
\LBcs(\Delta,d,2k-1)=2\binom{d+2}{2}-12\binom{d+2-2k}{2}+16\binom{d+2-3k}{2}
\]
and if $r=2k$ ($k>0$) then
\[
\LBcs(\Delta,d,2k)=2\binom{d+2}{2}-12\binom{d+1-2k}{2}+8\left(\binom{d-3k}{2}+\binom{d+1-3k}{2}\right).
\]
In Table~\ref{tbl:Cube} we list the values of $\LBcs(\Delta,d,r)$ for $1\le r\le 3$ and various $d$.  The actual values of $\dim \hspl^r_d(\Delta)$ appear in the \textit{gendim} column.  The numbers $D_\gamma$ for each value of $r$ appear in bold (they are the same as in Table~\ref{tbl:Bipyramid}).

\begin{table}[htbp]
	\centering
	\renewcommand{\arraystretch}{1.2}
	\begin{tabular}{cc}
		\begin{tabular}{c|c||c|c|c}
			$r$&$d$&$\binom{d+2}{2}$&$\LBcs(\Delta,d,r)$&\mbox{gendim}\\
			
			\hline
			
			1&\textbf{2}&    6       &      0     		&     6      \\
			1  & 3 &        10       &      0     		&    10      \\
			1  & 4 &    	15  	 &      6     		&    15      \\
			1  & 5 &    	21  	 &     18     		&    21      \\
			1  & 6 &    	28  	 &     36     		&    36      \\
			1  & 7 &    	    	 &     60     		&    60      \\
			1  & 8 &    	    	 &     90     		&    90      \\
			1  & 9 &    	    	 &    126     		&   126      \\
			
			\hline
			
			2&\textbf{3}&	10		 &       8          &     10     \\
			2  & 4 & 		15		 &       2          &     15     \\
			2  & 5 & 		21		 &       2          &     21     \\
			2  & 6 &  		28		 &       8          &     28     \\
			2  & 7 &    	36		 &      20          &     36     \\
			2  & 8 &   		45		 &      38          &     45     \\
		\end{tabular}
		& \quad
		\begin{tabular}{c|c||c|c|c}
			$r$&$d$&$\binom{d+2}{2}$&$\LBcs(\Delta,d,r)$&\mbox{gendim}\\
			
			\hline
			
			2  & 9 &  		55		 &      62          &     62     \\
			2  &10 & 				 &      92          &     92     \\
			2  &11 &   				 &     128          &    128     \\
			\hline
			
			3&\textbf{5}&  	21		 &       6          &     21     \\
			3  & 6 &  		28		 &       0          &     28     \\
			3  & 7 &  		36		 &       0          &     36     \\
			3  & 8 &  		45		 &       6          &     45     \\
			3  & 9 &  		55		 &      18          &     55     \\
			3  &10 &  		66		 &      36          &     66     \\
			3  &11 &  		78		 &      60          &     78     \\
			3  &12 &  		91		 &      90          &     91     \\
			3  &13 &  		105		 &     126          &    126     \\
			3  &14 &  				 &     168          &    168     \\
			3  &15 &  				 &     216          &    216     \\
			
			
			
		\end{tabular}
	\end{tabular}
	\vspace{5 pt}
	
	\caption{Bounds for the generic cube in Section~\ref{ss:cube}}\label{tbl:Cube}
\end{table}

\section{Concluding Remarks}\label{sec:ConcludingRemarks}
In this paper we have shown that the formula of Alfeld, Neamtu, and Schumaker in~\cite{ANS96} for homogeneous splines on closed tetrahedral vertex stars is a lower bound for $\dim \hspl^r_d(\Delta)$ when $d>D_\gamma$, where $D_\gamma$ is defined in~\eqref{eq:Dgamma}.  Using arguments due to Whiteley~\cite{WhiteleyComb} we have also shown that, for generic vertex positions, $\hspl^r_d(\Delta)$ consists only of global polynomials when $d<D_\gamma$.  

Our arguments suggest that, as in the planar case, the main obstruction to computing the dimension of the spline space on a vertex star is the nontrivial homology module of the Billera-Schenck-Stillman chain complex.  The contributions of this homology module are largely mysterious.  For instance, we see from Table~\ref{tbl:Bipyramid} that there are likely interesting contributions of this homology module to $\hspl^r_d(\Delta)$ for $r=3$, $d=7$ and $r=4$, $d=8,9$, where $\Delta$ is the non-generic bipyramid in Section~\ref{ss:nongen}.  These contributions are `unexpected' in the sense that we could not predict these jumps from either of the lower bounds in Section~\ref{ss:nongen}.  We did not find any example of a \textit{generic} closed vertex star which had similar behavior.  This leads us to the following question.

\begin{question}\label{ques:GenericClosedStars}
	If $\Delta$ is a generic closed vertex star and $d>D_\gamma$, is it true that 
	\[
	\dim \hspl^r_d(\Delta)=\max\left\lbrace\binom{d+2}{2},\LBcs(\Delta,d,r)\right\rbrace?
	\]
\end{question}

Surprisingly, it seems more difficult to pose the analog of Question~\ref{ques:GenericOpenStars} for open vertex stars.  Morally speaking, homogeneous splines on open tetrahedral vertex stars are indistinguishable from splines on planar triangulations, so we attempt to formulate Question~\ref{ques:GenericClosedStars} when $\Delta$ is a planar triangulation.  In this case $\LBos(\wDelta,d,r)\le \dim\spl^r_d(\Delta)$, where $\wDelta$ is the open vertex star obtained by coning over $\Delta$, and $\LBos(\wDelta,d,r)$ is simply Schumaker's lower bound from~\cite{SchumakerLower}.

One would like to ask the straightforward analog of Question~\ref{ques:GenericClosedStars}: If $\Delta$ is a generic triangulation, does $\dim \spl^r_d(\Delta)=\max\left\lbrace\binom{d+2}{2},\LBos(\wDelta,d,r)\right\rbrace$?  Unfortunately there are a few sub-configurations of $\Delta$ which can force this equality to fail.  We point out two of these, and would be curious to know if there are more.

First, suppose there is an edge in $\Delta$ both of whose vertices are on the boundary of $\Delta$; we call such an edge a \textit{chord} of $\Delta$.  A chord clearly gives rise to an extra spline of degree $r+1$ even for generic vertex positions.  Another configuration which gives rise to splines of low degree is the following: suppose $\sigma_1$ and $\sigma_2$ are adjacent triangles of $\Delta$ with vertices $\{\gamma_1,\gamma_2,\gamma\}$ and $\{\gamma_1,\gamma_2,\gamma'\}$, respectively.  We call $\sigma_1,\sigma_2$ a \textit{boundary pair} if the edges $\{\gamma_1,\gamma\}$ and $\{\gamma_1,\gamma'\}$ are both boundary edges of $\Delta$ or the edges $\{\gamma_2,\gamma\}$ and $\{\gamma_2,\gamma'\}$ are both boundary edges of $\Delta$.  If $\sigma_1,\sigma_2$ is a boundary pair then there is a spline on $\Delta$ supported only on $\sigma_1$ and $\sigma_2$ of degree $\lfloor3(r+1)/2\rfloor$.

\begin{question}\label{ques:GenericOpenStars}
	If $\Delta$ is a generic triangulation without a chord or a boundary pair, does
	\[
	\dim \spl^r_d(\Delta)=\max\left\lbrace\binom{d+2}{2},\LBos(\Delta,d,r)\right\rbrace
	\]
	for every $d\ge 0$?  If not, can the failure of equality be linked to a configuration like the chord or the boundary pair?
\end{question}

These questions are related to Schenck's `$2r+1$' conjecture~\cite{CohVan}, which states that $\dim \spl^r_d(\Delta)$ is given by Schumaker's lower bound (equivalently the graded Euler characteristic of $\calR/\calJ$) for $d\ge 2r+2$.  Recently Yuan and Stillman~\cite{YS19} found a counterexample to this conjecture, however they point out that the conjecture is still open for \textit{generic} triangulations.  If Schenck's conjecture is true for generic triangulations, then it implies that $\LBos(\wDelta,d,r)\ge \binom{d+2}{2}$ for $d\ge 2r+2$.  On the other hand, if Question~\ref{ques:GenericOpenStars} has a positive answer, then (modulo accounting for chords and boundary pairs) Schenck's conjecture for generic triangulations can be rephrased as: If $d\ge 2r+2$, then $\LBos(\Delta,d,r)\ge \binom{d+2}{2}$\,.  Checking this inequality simply amounts to estimating the roots of a quadratic polynomial.

\appendix

\section{}\label{app:1}

This appendix is devoted to the proof of Proposition~\ref{prop:5Vert}, the last remaining case of Theorem~\ref{thm:WhitelyGenericLowDegree}.  We direct the reader to~\cite{Comb} for unfamiliar terminology in the proof.

\begin{proposition}\label{prop:5Vert}
	If $\Delta$ is a generic closed tetrahedral vertex star with $f^\circ_1=5$ then $\dim \hspl^r_d(\Delta)=\binom{d+2}{2}$ for $d\le (5r+2)/3$.
\end{proposition}
\begin{proof}
	Lemma~\ref{lem:5Vert} shows that there is only one possibility for $G_\Delta$ (the graph on the left hand side in Figure~\ref{fig:5Vert}).  The corresponding closed tetrahedral star is a barycentric subdivision of a triangular bipyramid.  Thus we show that, for generic vertex positions, the barycentric subdivision of a triangular bipyramid has no non-trivial splines in degree $d\le (5r+2)/3$.
	
	The non-trivial $C^r$ splines on $\Delta$ are represented as the kernel of the map
	\[
	\bigoplus_{\sigma\in\Delta^\circ_2} \tJ(\sigma)\xrightarrow{\partial_2} \bigoplus_{\tau\in\Delta^\circ_1} \tJ(\tau).
	\]
	The graph $G_\Delta$ of the centrally triangulated triangular bipyramid is shown on the right in Figure~\ref{fig:5Vert}.  Orient the edge $\{i,j\}$ where $i<j$ by $i\to j$.  With this choice of orientation,  we can represent a tuple $G=(g_\tau\ell_\tau^{r+1})_{\tau\in\Delta^\circ_1}\in\ker \partial_2$ by the equations
	
	\begin{align}
	\label{e:1}
	-g_{12}\ell_{12}^{r+1}-g_{13}\ell_{13}^{r+1}-g_{14}\ell_{14}^{r+1}-g_{15}\ell_{15}^{r+1} & =0\\
	\label{e:2}
	g_{12}\ell_{12}^{r+1}-g_{23}\ell_{23}^{r+1}-g_{25}\ell_{25}^{r+1} & =0\\
	\label{e:3}
	g_{13}\ell_{13}^{r+1}+g_{23}\ell_{23}^{r+1}-g_{34}\ell_{34}^{r+1}-g_{35}\ell_{35}^{r+1} & =0\\
	\label{e:4}
	g_{14}\ell_{14}^{r+1}+g_{34}\ell_{34}^{r+1}-g_{45}\ell_{45}^{r+1} & =0\\
	\label{e:5}
	g_{15}\ell_{15}^{r+1}+g_{25}\ell_{15}^{r+1}+g_{35}\ell_{35}^{r+1}+g_{45}\ell_{45}^{r+1} & =0
	\end{align}
	
	The polynomials $g_{ij}$ are the \textit{smoothing cofactors} of the associated spline.  Suppose that $G=(g_\tau\ell_\tau^{r+1})_{\tau\in\Delta^\circ_1}\in\ker \partial_2$ is non-zero.  We will show that $\deg(G)>(5r+2)/3$.
	
	Notice first that each $g_{ij}$ appears in one of the equations~\eqref{e:1},~\eqref{e:2},~\eqref{e:3}, or~\eqref{e:4}.  Hence if $G\neq 0$ its constituents must satisfy one of the equations~\eqref{e:1},~\eqref{e:2},~\eqref{e:3}, or~\eqref{e:4} non-trivially.  Suppose that $G$ only satisfies~\eqref{e:5} trivially (i.e. $g_{15}=g_{25}=g_{35}=g_{45}=0$).  Then $G$ must still satisfy the previous equations.  Suppose one of $g_{12}$ or $g_{23}$ is non-zero; then by~\eqref{e:2} $g_{12}\ell_{12}^{r+1}-g_{23}\ell_{23}^{r+1}=0$ hence both $g_{12}$ and $g_{23}$ are non-zero.  Clearly in this case $g_{12}$ is a multiple of $\ell_{23}^{r+1}$ and $g_{23}$ is a multiple of $\ell_{12}^{r+1}$, hence $G$ has degree at least $2(r+1)>(5r+2)/3$.  Likewise if one of $g_{14}$ or $g_{34}$ is non-zero then both must be and $G$ has degree at least $2(r+1)$ by~\eqref{e:4}.  If $g_{12}=g_{23}=g_{14}=g_{34}=0$ in addition to $g_{15}=g_{25}=g_{35}=g_{45}=0$, then we can argue by~\eqref{e:1} or~\eqref{e:3} that $G$ will have degree at least $2(r+1)$ in the same way.
	
	Now suppose that $g_{14}=g_{34}=g_{45}=0$.  Then the spline $G$ restricts to a spline on the Alfeld split of a tetrahedron.  As before, if $G$ is non-trivial it must have degree at least $2r+1>(5r+2)/3$ by~\cite{S14}.
	
	So we may assume that $G$ satisfies both~\eqref{e:4} and~\eqref{e:5} non-trivially.  Furthermore we can assume that $g_{14},g_{34},g_{45}$ are all non-zero and at least two of $g_{15},g_{25},$ and $g_{35}$ are non-zero (otherwise we could repeat the argument above, yielding that $G$ has degree at least $2(r+1)$).  Notice that $g_{45}$ gives a non-zero element of the intersection
	\[
	I=\langle \ell_{34}^{r+1},\ell_{14}^{r+1}\rangle:\ell_{45}^{r+1}\cap \langle \ell_{15}^{r+1},\ell_{25}^{r+1},\ell_{35}^{r+1}\rangle:\ell_{45}^{r+1},
	\]
	where $:$ represents a \textit{colon ideal}.  That is, if $J$ is an ideal and $f$ a polynomial, $J:f$ is the ideal of all polynomials $g$ so that $fg\in J$.
	
	We claim that this intersection is empty in degrees $d\le (5r+2)/3$, which will complete the proof.  To prove this claim, we make a change of variables so that $\gamma_4$ points along the positive $z$-axis, $\gamma_5$ points along the positive $x$-axis, $\gamma_3$ points along the positive $y$-axis, and $\gamma_1$ points along the ray $(t,t,t)$ where $t<0$. Under this change of coordinates, the ideal $I$ becomes
	\[
	I=\langle x^{r+1},(x-y)^{r+1}\rangle:y^{r+1}\cap \langle (y-z)^{r+1},\ell_{25}^{r+1},z^{r+1}\rangle:y^{r+1},
	\]
	where $\ell_{25}$ is a linear form in the variables $y$ and $z$.  Put
	\[
	I_1=\langle x^{r+1},(x-y)^{r+1}\rangle:y^{r+1} \quad\mbox{and}\quad I_2=\langle (y-z)^{r+1},\ell_{25}^{r+1},z^{r+1}\rangle:y^{r+1}.
	\]
	In the rest of the proof we will show that the initial ideals $\mbox{in}(I_1)$ and $\mbox{in}(I_2)$ with respect to lexicographic order do not intersect in degrees $d\le (5r+2)/3$, which will also imply that $(I_1\cap I_2)_d=\emptyset$.
	
	Put $J_1=\langle x^{r+1},y^{r+1}\rangle:(x+y)^{r+1}$; $I_1$ can be obtained from $J_1$ by the change of coordinates $x\to x$, $y\to -x+y$.  In~\cite[Lemma~7.18]{RegSplines} it is shown that the initial ideal $\mbox{in}(J_1)$ with respect to lexicographic order consists of the $\dim (J_1)_d$ lexicographically largest monomials in the variables $x$ and $y$.  In other words, $\mbox{in}(J_1)$ in lexicographic order is a so-called \textit{lex segment} ideal (see~\cite[Chapter~2]{Comb}).  
	
	We claim that $\mbox{in}(I_1)$ is also a lex segment ideal.  To prove this claim we consider the effect of the change of coordinates $x\to x, y\to -x+y$ on $\mbox{in}(J_1)$ in degree $d$.  Under this change of coordinates, the vector space $\mbox{in}(J_1)_d$ becomes
	\[
	x^d, x^{d-1}(-x+y), x^{d-2}(-x+y)^2,\ldots,x^{d-a}(-x+y)^a,
	\]
	where $a+1=\dim (I_1)_d$.  Clearly the vector space spanned by these monomials is the same as the vector space spanned by
	\[
	x^d, x^{d-1}y, x^{d-2}y^2,\ldots,x^{d-a}y^a.
	\]
	It follows that $\mbox{in}(J_1)\subset\mbox{in}(I_1)$.  Since $I_1$ and $J_1$ have the same Hilbert function, we must in fact have $\mbox{in}(J_1)=\mbox{in}(I_1)$, so $\mbox{in}(I_1)$ is also a lex segment ideal.
	
	Finally, we use some information about the Hilbert functions of $I_1$ and $I_2$.  The degrees of syzygies of ideals in two variables generated by powers of linear forms are described explicitly in~\cite{MinReg} (uniform powers) and~\cite{FatPoints} (non-uniform powers).  From this analysis it follows that $\alpha(I_2)=\lfloor (r+1)/3 \rfloor$ (that is, the minimal generators of $I_2$ are in degrees at least $\lfloor (r+1)/3 \rfloor$).  Put $K=\lfloor (r+1)/3 \rfloor$ and $N=\langle y,z\rangle^K$.  Clearly $I_2\subset N$.  We show that $\alpha(\mbox{in}(I_2)\cap N)> (5r+2)/3$.
	
	It turns out that $I_1$ is a complete intersection generated in degrees $\lfloor (r+1)/2\rfloor,\lceil (r+1)/2\rceil$.  This implies that the Hilbert function of $I_1$ has the following form (for a proof see~\cite[Corollary~7.17]{RegSplines}):
	\[
	\dim (I_1)_d=\left\lbrace
	\begin{array}{ll}
	0 & 0\le d\le \lfloor (r+1)/2\rfloor\\
	2d+1-r & \lfloor (r+1)/2\rfloor\le d<r\\
	d+1 & d\ge r.
	\end{array}\right.
	\]
	Coupled with the fact that $\mbox{in}(I_1)$ is a lex segment ideal, we obtain that a monomial $x^ay^bz^c\in\mbox{in}(I_2)$ if and only if $2(a+b)+1-r\ge b+1$, or $2a+b\ge r$.  Similarly, $x^ay^bz^c\in N$ if and only if $b+c\ge \lfloor (r+1)/3 \rfloor$.
	
	Hence to find the least degree of a monomial in $\mbox{in}(I_1)\cap N$, we solve the integer linear program: minimize $a+b+c$ subject to $2a+b\ge r$ and $b+c\ge \lfloor (r+1)/3 \rfloor$.  Over the rationals, it is straightforward to check that this is minimal when $c=0$, $b=\lfloor (r+1)/3 \rfloor$, and $a= \frac{1}{2}(r-\lfloor (r+1)/3 \rfloor)$ with a value of $a+b+c=\frac{1}{2}\lfloor \frac{4r+1}{3}\rfloor$.  Thus $g_{45}$ must have degree greater than $\frac{1}{2}\lfloor \frac{4r+1}{3}\rfloor$, and so $g_{45}\ell_{45}^{r+1}$ must have degree greater than $r+1+\frac{1}{2}\lfloor \frac{4r+1}{3}\rfloor$.  To prove the statement of the Proposition, it suffices to show that $\frac{5r+2}{3}<r+1+\frac{1}{2}\lfloor \frac{4r+1}{3}\rfloor$, which is equivalent to $\frac{4r+1}{3}-1<\lfloor \frac{4r+1}{3}\rfloor$.  The last inequality is clearly true.
\end{proof}


\end{document}